\documentclass[11pt]{article}
\usepackage{amsmath,amssymb,textcomp,stmaryrd,xifthen,psfrag,graphicx,color}
\usepackage{amsfonts,amsthm}
\usepackage{color}

\SetSymbolFont{stmry}{bold}{U}{stmry}{m}{n} % get rid of the warning when using \boldsymbol ad stmry
\usepackage[T1]{fontenc}
\date{}

\newtheorem{theorem}{Theorem}
\newtheorem{lemma}[theorem]{Lemma}

\theoremstyle{definition}

\newcommand{\dual}[2]{\langle#1\hspace*{.5mm},#2\rangle}
\newcommand{\vdual}[2]{(#1\hspace*{.5mm},#2)}
\newcommand{\abs}[1]{\vert #1 \vert}
\newcommand{\norm}[3][]{#1\|#2#1\|_{#3}}
\newcommand{\snorm}[2]{|#1|_{#2}}

\newcommand{\est}{\mathrm{est}}
\newcommand{\err}{\mathrm{err}}

\newcommand{\diam}{\mathrm{diam}}
\newcommand{\wilde}{\widetilde}
\newcommand{\wat}{\widehat}

\newcommand{\jump}[1]{[#1]}
\newcommand{\avg}[1]{\{#1\}}
\newcommand{\hp}{\mathrm{hp}}

\def\eps{\varepsilon}

\newcommand{\TtoT}{\ensuremath{{\Theta}}}
\newcommand{\TtoTmat}{\ensuremath{\boldsymbol{\Theta}}}
\newcommand{\BLFmat}{\ensuremath{\mathbf{B}}}

\newcommand{\R}{\ensuremath{\mathbb{R}}}
\newcommand{\N}{\ensuremath{\mathbb{N}}}

\newcommand{\Dd}{\ensuremath{\mathcal{D}}}

\newcommand{\Ff}{\ensuremath{\mathcal{F}}}

\newcommand{\vv}{\ensuremath{\boldsymbol{v}}}
\newcommand{\ww}{\ensuremath{\boldsymbol{w}}}
\newcommand{\TT}{\ensuremath{\mathcal{T}}}
\newcommand{\Ss}{\ensuremath{\mathcal{S}}}
\newcommand{\el}{\ensuremath{T}}

\newcommand{\OO}{\ensuremath{\mathcal{O}}}

\newcommand{\VV}{\ensuremath{\mathbf{V}}}

\newcommand{\ff}{\ensuremath{\mathbf{f}}}
\newcommand{\xx}{\ensuremath{\mathbf{x}}}

\newcommand{\uu}{\boldsymbol{u}}

\newcounter{constantsnumber}
\def\setc#1{
  \ifthenelse{\equal{#1}{poinc}}{C_{\rm edge}}{ % estimating edge contributions
   \refstepcounter{constantsnumber}
   \label{const#1}C_{\theconstantsnumber}}}

\def\definec#1{\refstepcounter{constantsnumber}\label{const#1}}%

\def\c#1{
  \ifthenelse{\equal{#1}{infsup}}{C_{\rm infsup}}{ % estimating edge contributions
  \ifthenelse{\equal{#1}{stab}}{C_{\rm b}}{ % estimating edge contributions
    C_{\ref{const#1}}}}}

\title{DPG method with optimal test functions\\for a fractional advection diffusion equation
\thanks{Supported by CONICYT through FONDECYT projects 1150056, 3140614,
        3150012, and Anillo ACT1118 (ANANUM).}}
\author{Vincent~J.~Ervin\thanks{Department of Mathematical Sciences,
	  Clemson University, Clemson, South Carolina 29634-0975.
	  email: {\tt vjervin@clemson.edu}}
	\and Thomas~F\"uhrer$^\ddagger$ \and
	Norbert~Heuer$^\ddagger$ \and Michael~Karkulik
	\thanks{Facultad de Matem\'aticas, Pontificia Universidad Cat\'olica de Chile,
	  Avenida Vicu\~na Mackenna 4860, Macul, Santiago, Chile,
	  email: {\tt \{tofuhrer,nheuer,mkarkulik\}@mat.puc.cl}}}
\begin{document}
\maketitle
\begin{abstract}
  We develop an ultra-weak variational formulation of a fractional advection diffusion problem in
one space dimension and prove its well-posedness. Based on this formulation, we define
a DPG approximation with optimal test functions and show its quasi-optimal convergence.
Numerical experiments confirm expected convergence properties, for uniform and
adaptively refined meshes.


\end{abstract}
%===============================================================================================================
\section{Introduction}
%===============================================================================================================
In this paper we develop a discontinuous Petrov-Galerkin (DPG) method with optimal test
functions for a one-dimensional fractional advection diffusion problem of the form
\begin{align}\label{eq:model}
  \begin{split}
    -D D^{\alpha-2}D u + b Du + cu &= f \quad\text{ on } I := (0,1),\\
    u(0)=u(1)&=0.
  \end{split}
\end{align}
Here, $D$ denotes a single spatial derivative, and $D^{\alpha-2}$, for $\alpha\in(1,2)$, represents
a fractional integral operator of order $\alpha-2$. Throughout, we assume that
$c\in L^\infty([0,1])$, $b\in C^1([0,1])$, and $c-Db/2\geq 0$.
%For the ease of exposition, we will choose $D^{\alpha-2}$ to be the left Riemann-Liouville fractional
%integral operator, see Section~\ref{section:fracint} below.

Fractional advection diffusion equations have been receiving increased attention over the past decade as modeling 
equations for physical phenomena in such areas as contaminant transport in ground water flow \cite{ben001},
viscoelasticity \cite{mai971}, turbulent flow \cite{mai971, shl871}, and chaotic dynamics \cite{zas931}. As most
models involving fractional order differential equations do not have closed form solutions particular attention has been paid
to the development of numerical approximation schemes for these equations. Two phenomena of fractional order
differential equations which impact their numerical discretization and approximation are: (i) the fractional differential
operator is nonlocal (leading to a dense coefficient matrix), and (ii) the (typical) low regularity of the solution (leading to
slow convergence of the numerical solution to the true solution).

The first approximation methods investigated for fractional order differential equations were finite difference schemes
proposed by Liu, Ahn and Turner \cite{liu041}, and Meerschaert and Tadjeran \cite{mee041}, (see also
\cite{tad071, cui091, wan121}). Subsequently, finite element \cite{ErvinR_NumPDE_06, wan131, liu111, Xu141, JinLPR_Mcom_15} 
and spectral methods \cite{li121, wan151, zay151} have been developed for the approximation of 
fractional order differential equations. We note that a finite difference approximation using the Gr\"{u}nwald formula on
a uniform mesh leads to a Toeplitz like matrix which significantly reduces the storage required for the coefficient matrix,
and whose linear system can be very efficiently solved using a fast Fourier transform \cite{wan121}.

Fractional diffusion problems are inherently difficult to analyze and with our method we open a
way to deal with singularly perturbed cases (not considered here).
In fact, principal objective of the DPG method is to provide robust discretizations of singularly perturbed
problems like convection diffusion~\cite{DemkowiczH_13,BroersenS_14_IMA,ChanHBD_14,BramwellDGQ_12_elasticity}
and wave problems \cite{zmdgpc11}.
The DPG method with optimal test functions has been developed by Demkowicz, Gopalakrishnan and co-workers.
In its most common form it combines several ideas. These are ultra-weak variational formulations
(cf.~\cite{Despres_94_SUF,CessenatD_98_UWF}) with additional trace and flux
unknowns (cf.~\cite{BottassoMS_02_DPG}), and the utilization of specific test functions which
are designed for stability (cf. the SUPG method in~\cite{HughesB_79}
and test functions in~\cite{BarrettM_84}).
Demkowicz and Gopalakrishnan combine these
ideas in a discontinuous setting and by employing problem-tailored norms. Appropriately combined, the resulting
DPG method with optimal test functions delivers robust error control and also gives
access to localized a posteriori error estimation (or rather calculation).
For details we refer to \cite{dg11,DemkowiczG_11_CDP}.
In this paper we follow precisely these steps to deal with equations involving fractional diffusion.
By writing~\eqref{eq:model} as a first-order system, cf.~\eqref{eq:fos}, we develop an ultra-weak
variational formulation in Section~\ref{section:uw} below.
While a weak formulation of (\ref{eq:model})
leads to a non-symmetric, coercive bilinear form,
for the DPG method with optimal test functions the resulting variational formulation is always symmetric,
positive definite, implying existence of a unique solution.
This is the central result of the DPG method with optimal test functions,
stated below in Theorem~\ref{thm:dpg}.
Necessary conditions for its application are the well-known Babu\v{s}ka-Brezzi
conditions~\eqref{eq:bb}, which we check in Section~\ref{section:technicalresults}
for our ultra-weak formulation.
A central step will be to extend Riemann-Liouville fractional integral operators to negative
order Sobolev spaces and prove their ellipticity. To that end, we extend recent results
from~\cite{JinLPR_Mcom_15}.
In our main result, Theorem~\ref{thm:main}, we show well-posedness of the underlying
ultra-weak variational formulation and quasi-optimal convergence of the discrete scheme.
In particular, we will gain access to error control and adaptivity.
In Section~\ref{section:numerics}, we report on several numerical experiments that
illustrate convergence orders of variants with uniform meshes and with adaptively refined meshes.

We note that in \cite{WangYZ_PGFractional_CMAME_15} the authors propose a simplified Petrov-Galerkin method
with optimal test functions for fractional diffusion. They stick to discrete spaces with continuous functions
and calculate test functions globally. In contrast, we develop the fully discontinuous variant
that allows for local calculations of test functions. This is particularly important for fractional-order
problems where inner products are defined by double integrals so that global calculations are prohibitively costly.
Let us also mention that there is DPG-technology available for hypersingular integral
equations~\cite{HeuerP_DPG_SINUM_14,HeuerK_14_DPGBEM}. Hypersingular operators are of order one with energy spaces
of order $1/2$. For closed curves/surfaces, DPG theory can be established with integer-order Sobolev spaces and
is then simpler in a certain way. For open curves/surfaces however, one has to return to non integer-order spaces.
The case of hypersingular operators can be seen as a limit of fractional diffusion operators with orders
between one and two, as considered in this paper.

%===============================================================================================================
\section{Mathematical setting and main results}
%===============================================================================================================
We use the widespread notation $A \lesssim B$ to denote the fact that $A \leq C \cdot B$ where
$C>0$ does not depend on any quantities of interest.
By $A\simeq B$ we mean that both $A \lesssim B$ and $B \lesssim A$ hold.
Throughout, suprema are taken over the indicated sets \textit{except} 0.
%--------------------------------------------------------------------------------------------
\subsection{DPG method with optimal test functions}
%--------------------------------------------------------------------------------------------
We briefly recall the premises and results of the DPG method with optimal test
functions, cf.~\cite{dg11,DemkowiczG_11_CDP,zmdgpc11}. Given a Banach space
$U$, a Hilbert space $V$, and a bilinear form $b:U\times V\rightarrow \R$, we
consider the following three conditions:
\begin{subequations}\label{eq:bb}
  \begin{align}\label{eq:bb:1}
    b(\uu,\vv)= 0 \text{ for all } \vv\in V \implies \uu = 0;
  \end{align}
  there is a positive constant$\definec{infsup}$ $\c{infsup}$ such that
  \begin{align}\label{eq:bb:2}
    \c{infsup}\norm{\vv}{V} \leq \sup_{\uu\in U}
    \frac{b(\uu,\vv)}{\norm{\uu}{U}}\quad \text{ for all } \vv\in V;
  \end{align}
  there is a positive constant$\definec{stab}$ $\c{stab}$ such that
  \begin{align}\label{eq:bb:3}
    b(\uu,\vv)\leq \c{stab}\norm{\uu}{U}\norm{\vv}{V}
    \qquad\text{ for all } \uu\in U, \vv\in V.
  \end{align}
\end{subequations}
Define the so-called trial-to-test operator $\TtoT:U\rightarrow V$ by
\begin{align}\label{eq:ttot}
  \dual{\TtoT \uu}{\vv}_V = b(\uu,\vv)\quad \text{ for all }\vv\in V.
\end{align}
The following result is central to the DPG method and is, in the end, consequence of the Babu\v{s}ka-Brezzi
theory~\cite{b70,b74,xz03},~cf.~\cite{dg11} and related references given in the introduction.

\begin{theorem}\label{thm:dpg}
  Suppose that~\eqref{eq:bb:1}--\eqref{eq:bb:3} hold
  for a Banach space $U$, a Hilbert space $V$,
  and a bilinear form $b:U\times V\rightarrow \R$.
  Then, an equivalent norm on $U$ is given by
  \begin{align*}
    \norm{\uu}{E} := \sup_{\vv\in V}\frac{b(\uu,\vv)}{\norm{\vv}{V}},
    \quad\text{ and }\quad \c{infsup}\norm{\uu}{U} \leq \norm{\uu}{E}.
  \end{align*}
  Furthermore, for any $\ell\in V'$, the problem
  \begin{align}\label{thm:dpg:eq:weak}
    \text{ find } \uu\in U \text{ such that }\quad
    b(\uu,\vv) = \ell(\vv)\quad \text{ for all }\vv\in V
  \end{align}
  has a unique solution, and
  \begin{align}\label{thm:dpg:eq:stab}
    \norm{\uu}{E} \leq \norm{\ell}{V'}.
  \end{align}
  In addition, if $U_\hp\subset U$ is a finite-dimensional subspace, then the
  problem
  \begin{align}\label{thm:dpg:eq:discrete}
    \text{ find } \uu_\hp\in U_\hp \text{ such that }\quad
    b(\uu_\hp,\vv_\hp) = \ell(\vv_\hp)\quad \text{ for all }\vv_\hp\in \TtoT(U_\hp)
  \end{align}
  has a unique solution, and
  \begin{align}\label{thm:dpg:eq:approx}
    \norm{\uu-\uu_\hp}{E} = \inf_{\uu_\hp'\in
    U_\hp}\norm{\uu-\uu_\hp'}{E}.
  \end{align}
\end{theorem}
%--------------------------------------------------------------------------------------------
\subsection{Sobolev spaces}
%--------------------------------------------------------------------------------------------
For $s\in\R$ with $s\geq 0$ and an open interval $M=(a,b)\subseteq\R$, the Sobolev spaces $H^s(M)$
are defined via distributional derivatives and the Sobolev-Slobodeckij seminorm
$\snorm{\cdot}{H^s(M)}$ and norm $\norm{\cdot}{H^s(M)}$.
The space $\wilde H^s(M)$ is defined as the space of functions whose extension by zero is in $H^s(\R)$.
The space $H^{-s}(M)$ denotes the topological dual space of
$\wilde H^s(M)$, while $\wilde H^{-s}(M)$ denotes the dual of $H^s(M)$.
For a finite partition $\TT$ of $I=(0,1)$ into open, disjoint, and connected sets, we define
$H^s(\TT) := \prod_{\el\in\TT}H^s(\el)$, or, likewise,
$\wilde H^s(\TT) := \prod_{\el\in\TT}\wilde H^s(\el)$, with product norms
$\norm{v}{H^s(\TT)}^2 := \sum_{\el\in\TT}\norm{v|_\el}{H^s(\el)}^2.$
We also write $\wilde H^{-s}(\TT)$ or $H^{-s}(\TT)$ for the duals of product
spaces.
By $N:=\#\TT$ we denote the number of elements in the partition
and for $v\in H^s(\TT)$, $1/2<s$, we define the jump $\jump{v}\in\R^{N+1}$
as the vector of the differences of the traces of $v$ on the elements to the right and to the
left of all nodes $x=\overline{T_-}\cap \overline{T_+}$.
For the boundary nodes (i.e., $0$ and $1$), we just take traces.
For $v\in H^s(\TT)$, $1/2<s$, we also define the average $\avg{v}$ 
as the vector of mean values of the traces of $v$ on the elements to the right and to the left
of all nodes.
We will need certain results for this kind of spaces. 
From now on, we assume that partitions are quasi-uniform, i.e.,
for all $T\in\TT$ holds $\snorm{T}{}\simeq N^{-1}$ for $N:=\#\TT$ being
the number of elements in the partition $\TT$, and the constant involved
is independent of $\TT$.
We denote by $D_\TT$ the $\TT$-piecewise distributional derivative.
\begin{lemma}\label{lemma:scaling}
  The following statements hold with constants which only depend on $s$:
  \begin{itemize}
    \item Let $s\in(0,1/2)$. There holds
      \begin{align}\label{lemma:scaling:1}
        \norm{v}{H^s(I)} \lesssim N^{s}\norm{v}{H^{s}(\TT)}
        \qquad\text{ for all }v\in H^s(I).
      \end{align}
    \item Let $s\in(1/2,1)$. There holds
      \begin{align}\label{lemma:scaling:2}
        \norm{D_\TT v}{\wilde H^{s-1}(I)} \lesssim N^{1-s} \snorm{v}{H^s(\TT)}
        \qquad\text{ for all } v\in H^s(\TT).
      \end{align}
    \item Let $s\in (1/2,1]$. There holds
      \begin{align}\label{lemma:scaling:3}
	\abs{\jump{v}} \lesssim N^{1/2} \norm{v}{H^s(\TT)}
        \qquad\text{ for all } v\in H^s(\TT).
      \end{align}
  \end{itemize}
\end{lemma}
\begin{proof}
  The first statement is seen as follows: First, for $\wat T$ a reference
  interval with fixed diameter, there is a constant $C_s>0$ such that
  $\norm{\wat v}{\wilde H^s(\wat T)} \leq C_s \norm{\wat v}{H^s(\wat T)}$,
  cf.~\cite{Grisvard_EllipticProblems_85} and~\cite[Proof of
  Lemma~5]{Heuer_ASMSLO_NumerMath_01}. Second, scaling arguments show
  that $\norm{v}{\wilde H^s(T)} \lesssim N^s \norm{v}{H^s(T)}$ for all
  $T\in\TT$. Now,
  \begin{align}\label{eq:broken}
    \norm{v}{H^s(I)}^2 \lesssim \norm{v}{\wilde H^{s}(I)}^2
    \lesssim \sum_{T\in\TT}\norm{v}{\wilde H^s(T)}^2,
  \end{align}
  where the second estimate follows
  from~\cite[Lemma~20]{FeischlFHKP_ABEM_ARCME_14}.
  To show the second statement, we proceed as before and use an affine transformation
  on every element $T\in\TT$,
  \begin{align*}
    \norm{Dv}{\wilde H^{s-1}(T)}^2 \lesssim N \norm{\wat D \wat v}{\wilde H^{s-1}(\wat T)}^2
    \lesssim N \norm{\wat D\wat v}{H^{s-1}(\wat T)}^2.
  \end{align*}
  Here the second estimate follows as the norms involved are dual
  to norms on which we can use~\cite{Grisvard_EllipticProblems_85}
  and~\cite[Proof of Lemma~5]{Heuer_ASMSLO_NumerMath_01}. A quotient space
  argument on the reference element $\wat T$, cf.~\cite{Heuer_Equivalence_JMAA_14}, shows
  \begin{align*}
    \norm{\wat D\wat v}{H^{s-1}(\wat T)} \lesssim \inf_{c\in\R} \norm{\wat v+c}{H^s(\wat T)}
    \simeq \snorm{\wat v}{H^s(\wat T)}.
  \end{align*}
  The second statement follows by application of the scaling argument
  $\snorm{\wat v}{H^s(\wat T)}^2 \lesssim N^{1-2s}\snorm{v}{H^s(T)}^2$.
  The third statement follows easily from, e.g.,
  \begin{align*}
    \abs{v(x)} \leq \norm{v}{L^\infty(\wat T)} \lesssim \norm{\wat v}{H^s(\wat T)}
    \lesssim N^{1/2}\norm{v}{H^s(T_+)}.
  \end{align*}
  Here, for example, $x = \overline{T_-}\cap \overline{T_+}$, the second estimate follows
  by the Sobolev Embedding theorem, and the third one again by a scaling argument.
\end{proof}

%--------------------------------------------------------------------------------------------
\begin{lemma}\label{lemma:poincare:pw}
  There holds
  \begin{align*}
    \norm{\tau}{L_2(I)} \lesssim \norm{D_\TT \tau}{L_2(I)} + N^{1/2}\abs{\jump{\tau}}
    \qquad\text{ for all } \tau\in H^1(\TT),
  \end{align*}
  and the hidden constant is independent of $\TT$.
\end{lemma}
\begin{proof}
  Let $\phi\in \wilde H^1(I)$ be the weak solution of $-D^2 \phi=\tau$.
  Then $D\phi \in H^1(I)$ with distributional derivative $D^2\phi=-\tau$, and
  integration by parts yields
  \begin{align*}
    \vdual{\tau}{\tau} = - \vdual{\tau}{D^2\phi} = \vdual{D_\TT \tau}{D\phi}
     + \dual{\jump{\tau}}{D\phi}.
%    - \sum_{j=1}^N \left( \tau(x_j)D\phi(x_j) - \tau(x_{j-1})D\phi(x_{j-1})
%    \right),
  \end{align*}
  Cauchy-Schwarz and Lemma~\ref{lemma:scaling}, eq.~\eqref{lemma:scaling:3} imply
  \begin{align*}
    \norm{\tau}{L_2(I)}^2
    %    \norm{D_\TT\tau}{L_2(I)} \norm{D\phi}{L_2(I)}
%    + \abs{\jump{\tau}} \norm{D\phi}{C(\overline{I})}
    \lesssim \norm{D_\TT\tau}{L_2(I)} \norm{D\phi}{L_2(I)}
    + N^{1/2}\abs{\jump{\tau}} \norm{D\phi}{H^1(I)}.
  \end{align*}
  By construction, $\norm{D\phi}{H^1(I)} \lesssim \norm{\tau}{L_2(I)}$, which
  concludes the proof.
\end{proof}
%--------------------------------------------------------------------------------------------
We will need the following result on fractional seminorms.
\begin{lemma}\label{lemma:poincare}
  Let $s\in (0,1)$ be fixed. There holds
  \begin{align*}
    \snorm{u}{H^s(I)} \lesssim \norm{Du}{H^{s-1}(I)}
    \quad\text{ for all } u\in H^s(I).
  \end{align*}
  where $Du$ is the distributional derivative of $u$. The hidden
  constant does not depend on $I$.
\end{lemma}
\begin{proof}
  As $u\in L_2(I)$, it holds $Du\in H^{-1}(I)$.
  We can write $u=D\psi + c$ with $c\in\R$ and $\psi\in \wilde H^1(I)$,
  where $\norm{\psi}{\wilde H^1(I)}\lesssim\norm{u}{L_2(I)}$.
  Due to the definition of the distributional derivative we see
  \begin{align*}
    \abs{\vdual{u}{D\psi}} = 
    \abs{\vdual{Du}{\psi}} \lesssim \norm{Du}{H^{-1}(I)}\norm{\psi}{\wilde H^1(I)}
    \lesssim \norm{Du}{H^{-1}(I)} \norm{u}{L_2(I)}
  \end{align*}
  We conclude that for $u\in L_2(I)$, it holds
  \begin{align*}
    \norm{u}{L_2(I)}^2 = \vdual{u}{D\psi} + \vdual{u}{c}
    \lesssim \norm{Du}{H^{-1}(I)} \norm{u}{L_2(I)} + \vdual{u}{c}.
  \end{align*}
  Now we apply this estimate to $u-\overline u$, where $\overline u$ denotes the mean value
  of $u$, and obtain
  \begin{align}\label{lemma:poincare:-1}
    \norm{u-\overline u}{L_2(I)} \lesssim \norm{Du}{H^{-1}(I)}.
  \end{align}
  The standard Poincar\'e inequality states that
  \begin{align}\label{lemma:poincare:0}
    \norm{u-\overline u}{H^1(I)} \lesssim \norm{Du}{L_2(I)}.
  \end{align}
  The $H^s(I)$ norm can equivalently be obtained by the K-method of
  interpolation, cf.~\cite{Triebel_InterpolationSpaces}, via
  \begin{align*}
    \norm{u-\overline u}{H^s(I)}^2 \simeq \norm{u-\overline
    u}{[L_2(I),H^1(I)]_{s,2}}^2 =
    \int_0^\infty t^{-2s} \left(
    \inf_{v\in H^1(I)} \norm{u-\overline{u}-v}{L_2(I)} + t \norm{v}{H^1(I)}
    \right)^{2} \frac{dt}{t}.
  \end{align*}
  Using~\eqref{lemma:poincare:-1} and~\eqref{lemma:poincare:0}, we obtain
  \begin{align*}
    \inf_{v\in H^1(I)} \norm{u-\overline{u}-v}{L_2(I)} + t \norm{v}{H^1(I)}
    &\leq
    \inf_{\substack{v\in H^1(I)\\\overline v=0}} \norm{u-\overline{u}-v}{L_2(I)} + t
    \norm{v}{H^1(I)}\\
    &\lesssim
    \inf_{\substack{v\in H^1(I)\\\overline v=0}}
    \norm{Du-Dv}{H^{-1}(I)} + t \norm{Dv}{L_2(I)}
  \end{align*}
  Next we use that for $w\in L_2(I)$ there is a $\psi\in H^1(I)$ with $\overline\psi=0$ such that
  $D\psi=w$. We conclude
  \begin{align*}
    \norm{u-\overline u}{H^s(I)}^2 \lesssim
    \int_0^\infty t^{-2s} \left(
    \inf_{w\in L_2(I)} \norm{Du-w}{H^{-1}(I)} + t
    \norm{w}{L_2(I)}
    \right)^{2} \frac{dt}{t}.
  \end{align*}
  By definition, the right-hand side is
  $\norm{Du}{[H^{-1}(I),L_2(I)]_{s,2}}^2$, which
  is equivalent to $\norm{Du}{H^{s-1}(I)}^2$.
  This concludes the proof for a specific $I$. The hidden constant does
  not depend on $I$, which can be shown by scaling arguments.
\end{proof}
%--------------------------------------------------------------------------------------------
\subsection{Fractional integral operators}\label{section:fracint}
%--------------------------------------------------------------------------------------------
The fractional integral operators that we will use are of so-called Riemann-Liouville type.
For $\beta>0$ we denote by $_0D^{-\beta}$ and $D_1^{-\beta}$ 
the left and right-sided versions of these operators, defined on $I=(0,1)$ by
\begin{align*}
  _0D^{-\beta} u (x) := \frac{1}{\Gamma(\beta)}\int_0^x (x-s)^{\beta-1}u(s)\;ds
  \quad\text{ and }\quad
  D_1^{-\beta} u (x) := \frac{1}{\Gamma(\beta)}\int_x^1 (s-x)^{\beta-1}u(s)\;ds.
\end{align*}
We also abbreviate $D^{-\beta} := {_0}D^{-\beta}$.
A standard textbook on this kind of operators is~\cite{SamkoKM_Fractional}.
Recently, classical results regarding boundedness and ellipticity of these operators
were extended in~\cite{ErvinR_NumPDE_06,JinLPR_Mcom_15}.
In order to obtain a variational formulation suited for DPG analysis,
we need to extend these operators to negative order Sobolev spaces and
show their ellipticity. To this end, let
$\Ff$ denote the Fourier transforms on the space $\Ss'(\R)$ of
tempered distributions, cf.~\cite[Chapter 7]{Rudin_FANA}, defined by
\begin{align*}
  \Ff u(\xi) := (2\pi)^{-1/2}\int_\R u(x)e^{-ix\xi}\,dx.
\end{align*}
%We need the following result from~\cite{JinLPR_Mcom_15}.
%\begin{lemma}\label{lemma:rli:L2}
%  For any $s,\beta\geq 0$, the operators $D_x^{-\beta}$ and $_xD^{-\beta}$
%  are bounded maps from $\wilde H^s(0,1)$ into $H^{s+\beta}(0,1)$.
%\end{lemma}
Choosing a space of test functions which is invariant under the action of $D^{-\beta}$ and
$D^{-\beta}$, these operators can be extended to the associated spaces of distributions,
cf.~\cite[\S8]{SamkoKM_Fractional}. In the present setting, a different argument can
be used.
\begin{lemma}\label{lemma:extension}
  For every $s\in\R$ with $-\beta\leq s$ and $\beta>0$, the operator $D^{-\beta}$
  can be extended to a bounded linear operator
  $D^{-\beta}:\wilde H^{s}(I)\rightarrow H^{s+\beta}(I)$.
\end{lemma}
\begin{proof}
  For $0 \leq s$, the statement was shown for $_0D^{-\beta}$ and $D_1^{-\beta}$
  in Theorem~3.1 of~\cite{JinLPR_Mcom_15}.
  It therefore remains to consider $-\beta\leq s<0$. We will show the statement for
  $s=-\beta$, the remaining cases follow by interpolation.
  We already know that $D_1^{-\beta}:L_2(I)\rightarrow H^{\beta}(I)$ is a linear and bounded
  operator.
  According to~\cite[Corollary of Thm.~3.5]{SamkoKM_Fractional}, it holds that
  \begin{align}\label{lemma:extension:eq1}
    \vdual{D^{-\beta}u}{v} = \vdual{u}{{D_1^{-\beta}v}}
    \quad \text{ for all } u,v\in L_2(I).
  \end{align}
  Hence, the right-hand side of~\eqref{lemma:extension:eq1} extends
  $D^{-\beta}$ to a linear, bounded operator $D^{-\beta}:\wilde H^{-\beta}\rightarrow L_2(I)$.
%  It is known, cf XXX, that $H^s(\R)$ consists of all $f\in \Ss'(\R)$ with
%  $\left( 1+\xi^2 \right)^{s/2}\Ff(f)\in L_2(\R)$. It holds $\wilde H^s(I)\subset H^s(\R)$
%  and we define for $f\in \wilde H^s(I)$
%  \begin{align*}
%    _aD_x^{-\beta}f := \Ff^{-1}\left( (i\xi)^{-\beta}\Ff(f) \right) \in \Ss'(\R).
%  \end{align*}
%  We conclude
%  \begin{align*}
%    (1+\xi^2)^{(s+\beta)/2}\Ff(_aD_x^{-\beta}f) \in L_2(\R),
%  \end{align*}
%  which means $_aD_x^{-\beta}f\in H^{s+\beta}(\R)$. For any value of $s+\beta$ it holds that
%  $H^{s+\beta}(\R)$ is continuously embedded in $H^{s+\beta}(I)$. By construction,
%  $_aD_x^{-\beta}$ is an extension of the classical definition.
\end{proof}
%--------------------------------------------------------------------------------------------
\begin{lemma}\label{lemma:rli:elliptic}
  The operator $D^{-\beta}$ is elliptic on $H^{-\beta/2}(\TT)$ for
  $0<\beta<1$.
\end{lemma}
\begin{proof}
  For a test function $\varphi\in \Dd(I)$ holds $\Ff(D^{-\beta}\varphi)(\xi)
  = (i\xi)^{-\beta}\Ff(\varphi)(\xi)$, cf.~\cite[Thm.~7.1]{SamkoKM_Fractional}.
  Then, a short computation (cf.~\cite[Proof of Lemma 2.4]{ErvinR_NumPDE_06})
  shows
%  The proof follows from the definition of $D^{-\beta}$ from
%  Lemma~\ref{lemma:extension}. Suppose $f\in H^{-\beta/2}(\TT)$.
%  Then $f\in \wilde H^{-\beta/2}(I)$. By definition
  \begin{align*}
    \vdual{D^{-\beta}\varphi}{\varphi}
    &= \vdual{(i\xi)^{-\beta}\Ff(\varphi)}{\overline{\Ff(\varphi)}}
    = \vdual{(i\xi)^{-\beta/2}\Ff(\varphi)}{\overline{(-i\xi)^{\beta/2}\Ff(\varphi)}}\\
    &=
    \cos(-\pi\beta/2)\vdual{(i\xi)^{-\beta/2}\Ff(\varphi)}{\overline{(i\xi)^{\beta/2}\Ff(\varphi)}}\\
    &\quad+ i\sin(-\pi\beta/2)
    \Bigl(
    \int_0^\infty (i\xi)^{-\beta/2}\Ff(\varphi)
    \overline{(i\xi)^{-\beta/2}\Ff(\varphi)}d\xi\\
    &\qquad\qquad\qquad\qquad- \int_{-\infty}^0 (i\xi)^{-\beta/2}\Ff(\varphi)
    \overline{(i\xi)^{-\beta/2}\Ff(\varphi)}d\xi
    \Bigr)
%    = \vdual{(i\xi)^{-\beta/2}\Ff(\varphi)}{(i\xi)^{-\beta/2}\Ff(\varphi)}\\
%    &=\norm{\xi^{-\beta/2}\Ff(\varphi)}{L_2(\R)}^2
%    \gtrsim \norm{(1+\xi^2)^{-\beta/4}\Ff(\varphi)}{L_2(\R)}^2
%    = \norm{\varphi}{H^{-\beta/2}(\R)}^2
  \end{align*}
  As the left-hand side of this identity is real, the imaginary part on the
  right-hand side vanishes. Furthermore, $\cos(-\pi\beta/2)>0$ for
  $0<\beta<1$. We obtain
  \begin{align*}
    \vdual{D^{-\beta}\varphi}{\varphi} \gtrsim
    \norm{(\xi^2)^{-\beta/4}\Ff(\varphi)}{L_2(\R)}^2
    \gtrsim
    \norm{(1+\xi^2)^{-\beta/4}\Ff(\varphi)}{L_2(\R)}^2.
  \end{align*}
  The right-hand side is equivalent to the norm
  $\norm{\varphi}{H^{-\beta/2}(\R)}$. A density argument shows the ellipticity
  on $H^{-\beta/2}(\R)$.
  Since on $H^{-\beta/2}(\TT)$ it holds
  $\norm{\cdot}{H^{-\beta/2}(\R)} \gtrsim \norm{\cdot}{H^{-\beta/2}(I)} \gtrsim
  \norm{\cdot}{H^{-\beta/2}(\TT)}$, cf.~\eqref{eq:broken},
  the proof is finished.
\end{proof}
%--------------------------------------------------------------------------------------------
\subsection{Ultra-weak formulation and main result}\label{section:uw}
%--------------------------------------------------------------------------------------------
%---------------------------------------------------------------------------------------------------------------
%\subsection{The case $\alpha>3/2$}
%---------------------------------------------------------------------------------------------------------------
We write~\eqref{eq:model} as first-order system
\begin{align}\label{eq:fos}
  \begin{split}
    \sigma - Du &= 0,\\
    -D D^{\alpha-2} \sigma +b Du + cu &= f.
  \end{split}
\end{align}
Then, we multiply these equations with $\tau$ respectively $v$, integrate by parts piecewise on a
partition $\TT$ and rename the appearing boundary terms of $D^{\alpha-2}\sigma$ and $u$
by $\wat \sigma$ and $\wat u$ to obtain
\begin{subequations}\label{eq:bf}
\begin{align}   
  \vdual{\sigma}{\tau} + \vdual{u}{D_\TT\tau} - \dual{\wat u}{\jump{\tau}} &= 0\label{eq:bf:a}\\
  \vdual{D^{\alpha-2}\sigma}{D_\TT v} + \vdual{b\sigma}{v} + \vdual{cu}{v} - \dual{\wat\sigma}{\jump{v}} &= \vdual{f}{v}\label{eq:bf:b}.
\end{align}
\end{subequations}
The left and right-hand sides of the preceding equations define our bilinear form
and linear form via
\begin{align*}
  b(\uu,\vv) := b(\sigma,u,\wat\sigma,\wat u;\tau,v) &:= 
  \vdual{\sigma}{\tau + D^{(\alpha-2)\star}D_\TT v + bv} + \vdual{u}{D_\TT\tau + cv}
  - \dual{\wat u}{\jump{\tau}}
  - \dual{\wat\sigma}{\jump{v}},\\
  \ell(\vv) := \ell(\tau,v) &:= \vdual{f}{v}.
\end{align*}
Here and from now on, $D^{(\alpha-2)\star}:\wilde H^{\alpha/2-1}(I)\rightarrow H^{1-\alpha/2}(I)$ 
denotes the conjugate of $D^{\alpha-2}$.
Define $U_\alpha:=\wilde H^{\alpha/2-1}(I)\times L_2(I)\times\R^{N+1}\times\R^{N-1}$ and
$V_\alpha:=H^1(\TT)\times H^{\alpha/2}(\TT)$, where $N$ is the number of elements of $\TT$,
with product norms
\begin{align*}
  \norm{\uu}{U_\alpha}^2 &:= \norm{\sigma}{\wilde H^{\alpha/2-1}(I)}^2 + \norm{u}{L_2(I)}^2
  + N^{-3}( \abs{\wat\sigma}^2 + \abs{\wat u}^2 ), \text{ and }\\
  \norm{\vv}{V_\alpha}^2 &:= \norm{\tau}{H^1(\TT)}^2 + \norm{v}{H^{\alpha/2}(\TT)}^2.
\end{align*}
By $\abs{\cdot}$, we mean the usual Euclidean norm. Our ultra-weak formulation now reads as follows:
given $\ell\in V_\alpha'$, we aim to find $\uu\in U_\alpha$ such that
\begin{align}\label{eq:equation}
  b(\uu,\vv) = \ell(\vv) \quad\text{ for all } \vv\in V_\alpha.
\end{align}
For a discrete subspace $U_\hp\subset U_\alpha$, the DPG method with
optimal test functions is to find $\uu_\hp\in U_\hp$ such that
\begin{align}\label{eq:equation:discrete}
  b(\uu_\hp,\vv_\hp) = \ell(\vv_\hp) \quad\text{ for all } \vv_\hp\in \TtoT_\alpha(U_\hp),
\end{align}
where $\TtoT_\alpha:U_\alpha\rightarrow V_\alpha$ is the trial-to-test operator
associated with $b$, cf.~\eqref{eq:ttot}.
The following theorem is the main result of this work. It states unique
solvability and stability of the continuous and discrete
formulations~\eqref{eq:equation} and~\eqref{eq:equation:discrete},
as well as a best approximation result.
\begin{theorem}\label{thm:main}
  For $\alpha\in(1,2)$, $f\in L_2(I)$, and arbitrary partition $\TT$,
  the variational formulation~\eqref{eq:equation} has a unique solution $\uu\in
  U_\alpha$, and
  \begin{align*}
    \norm{\uu}{U_\alpha} \lesssim \norm{f}{L_2(I)}.
  \end{align*}
  Furthermore, the discrete problem~\eqref{eq:equation:discrete} has a unique
  solution $\uu_\hp\in U_\hp$, and
  \begin{align*}
    \norm{\uu-\uu_\hp}{U_\alpha}
    \lesssim
    \inf_{(\sigma_\hp',u_\hp',\wat\sigma_\hp',\wat u_\hp') \in U_\hp}
    \left( 
    N^{1-\alpha/2}\norm{\sigma - \sigma_\hp'}{\wilde H^{\alpha/2-1}(I)}
    + \norm{u - u_\hp'}{L_2(I)}
    \right).
  \end{align*}
\end{theorem}
\begin{proof}
  We are going to apply Theorem~\ref{thm:dpg}, hence we
  check~\eqref{eq:bb:1}--\eqref{eq:bb:3}. The condition~\eqref{eq:bb:1}
  follows from Lemma~\ref{lemma:b:injective}. The condition~\eqref{eq:bb:3}
  follows from Lemma~\ref{lemma:b:cont}. It remains to check
  condition~\eqref{eq:bb:2}. To that end, observe first that
  \begin{align}\label{thm:main:eq:1}
%    \norm{\uu}{E_\alpha} = 
    \sup_{\uu\in U_\alpha} \frac{b(\uu,\vv)}{\norm{\uu}{U_\alpha}}
    = \left( 
    \norm{\tau+D^{(\alpha-2)\star}D_\TT v+bv}{H^{1-\alpha/2}(I)}^2 + 
    \norm{D_\TT\tau + cv}{L_2(I)}^2
    + N^{3}(\abs{\jump{\tau}} + \abs{\jump{v}})^2
    \right)^{1/2}
  \end{align}
  For given $\vv = (\tau,v)\in V_\alpha$ we define
  $\tau_1\in H^1(I)$ and $v_1\in \wilde H^{\alpha/2}(I)$ as the solution
  of Lemma~\ref{lemma:adjoint:inhom} with data $F:=D_\TT \tau + cv$ and
  $G = \tau + D^{(\alpha-2)\star}D_\TT v + bv$, and write
  $\tau = \tau_0 + \tau_1$ and $v = v_0 + v_1$. The functions $\tau_0$ and
  $v_0$ then fulfill the assumptions of Lemma~\ref{lemma:adjoint:hom}. The
  triangle inequality and Lemmas~\ref{lemma:adjoint:inhom} and~\ref{lemma:adjoint:hom}
  show
  \begin{align}\label{thm:main:eq:2}
    \norm{\vv}{V_\alpha}
    \lesssim 
    \norm{\tau+D^{(\alpha-2)\star}D_\TT v+bv}{H^{1-\alpha/2}(I)} + 
    \norm{D_\TT\tau + cv}{L_2(I)}
    + N^{3/2}(\abs{\jump{\tau}} + \abs{\jump{v}}).
  \end{align}
  The equations~\eqref{thm:main:eq:1} and~\eqref{thm:main:eq:2} show
  condition~\eqref{eq:bb:2}. Theorem~\ref{thm:dpg} shows
  that there are unique solutions $\uu$ and $\uu_\hp$
  of the problems~\eqref{eq:equation} and~\eqref{eq:equation:discrete}
  which fulfill stability~\eqref{thm:dpg:eq:stab} and
  best approximation~\eqref{thm:dpg:eq:approx}, and that
  \begin{align*}
    \c{infsup}\norm{\uu}{U_\alpha} \leq \norm{\uu}{E_\alpha}
    := \sup_{\vv\in V_\alpha}\frac{b(\uu,\vv)}{\norm{\vv}{V_\alpha}}.
  \end{align*}
  Lemma~\ref{lemma:b:cont} shows that
  \begin{align*}
    \inf_{\uu_\hp'\in U_\hp}&\norm{\uu-\uu_\hp'}{E_\alpha} \lesssim\\
    &\inf_{(\sigma_\hp',u_\hp',\wat\sigma_\hp',\wat u_\hp') \in U_\hp}
    \left( 
    N^{1-\alpha/2}\norm{\sigma-\sigma_\hp'}{\wilde H^{\alpha/2-1}(I)}
    + \norm{u-u_\hp'}{L_2(I)}
    \right).
  \end{align*}
  Here, owing to the fact that $\wat u$ and $\wat\sigma$ are just
  finite-dimensional vectors, the functions $\wat\sigma_\hp'$
  and $\wat u_\hp'$ can be omitted on the right-hand side.
\end{proof}
%---------------------------------------------------------------------------------------------------------------
%===============================================================================================================
\section{Technical results}\label{section:technicalresults}
%===============================================================================================================
The first lemma states boundedness of the bilinear form $b$.
\begin{lemma}\label{lemma:b:cont}
  For $\alpha\in(1,2)$,
  \begin{align*}
    \abs{b(\uu,\vv)}
%    \leq &\norm{\sigma}{\wilde H^{\alpha/2-1}(I)} \cdot
%    \norm{\tau+D^{(\alpha-2)\star}D_\TT v + bv}{H^{1-\alpha/2}(I)}
%    + \norm{u}{L_2(I)} \cdot \norm{D_\TT \tau + bv}{L_2(I)}\\
%    &+ N^{1/2}\abs{\wat u}\cdot
%    \norm{\tau}{H^1(\TT)}
%    +  N^{1/2}\abs{\wat \sigma}\cdot
%    \norm{v}{H^{\alpha/2}(\TT)}\\
    &\lesssim \left( N^{2-\alpha}\norm{\sigma}{\wilde
      H^{\alpha/2-1}(I)}^2 + \norm{u}{L_2(I)}^2 + N \abs{\hat u}^2 + N
      \abs{\hat \sigma}^2 \right)^{1/2} \norm{\vv}{V_\alpha},
  \end{align*}
  with a constant independent of $\TT$.
  In particular, $\abs{b(\uu,\vv)} \leq \c{stab}\norm{\uu}{U_\alpha}\norm{\vv}{V_\alpha}$,
  where the constant $\c{stab}$ depends on $N$.
\end{lemma}
\begin{proof}
  By Lemma~\ref{lemma:scaling}, eq.~\eqref{lemma:scaling:3}, we have
  \begin{align*}
    \abs{\dual{\wat u}{\jump{\tau}}} \lesssim N^{1/2}\abs{\wat u}
    \norm{\tau}{H^1(\TT)}\quad\text{ and } \quad
    \abs{\dual{\wat \sigma}{\jump{v}}} \lesssim N^{1/2}\abs{\wat \sigma}
    \norm{v}{H^{\alpha/2}(\TT)}.
  \end{align*}
  The triangle inequality, Lemmas~\ref{lemma:scaling}
  and~\ref{lemma:extension}, the definition of $D^{(\alpha-2)\star}$,
  $c\in L^\infty([0,1])$, $b\in C^1([0,1])$, and $1\leq\alpha$ show
  \begin{align*}
    \norm{\tau+D^{(\alpha-2)\star}D_\TT v + bv}{H^{1-\alpha/2}(I)}
    &\lesssim \norm{\tau}{H^{1-\alpha/2}(I)}
    + \norm{D_\TT v}{\wilde H^{\alpha/2-1}(I)}
    + \norm{bv}{H^{1-\alpha/2}(I)}\\
    &\lesssim N^{1-\alpha/2}\left( 
    \norm{\tau}{H^1(\TT)} + \norm{v}{H^{\alpha/2}(\TT)}
  \right)
  \end{align*}
  and
  \begin{align*}
    \norm{D_\TT \tau + cv}{L_2(I)} \lesssim \norm{\tau}{H^1(\TT)} + \norm{v}{H^{\alpha/2}(\TT)}.
  \end{align*}
  We finish the proof with the triangle and Cauchy-Schwarz inequalities.
\end{proof}
%---------------------------------------------------------------------------------------------------------------
\begin{lemma}\label{lemma:b:injective}
  It holds that
  \begin{align*}
    b(\uu,\vv) = 0 \text{ for all } \vv\in\VV\Longleftrightarrow \uu=0.
  \end{align*}
\end{lemma}
\begin{proof}
  The direction $\Leftarrow$ is clear, and we proceed with the implication $\Rightarrow$.
  Using $\tau\in C_0^\infty(I)$ in~\eqref{eq:bf:a} shows that the distributional derivative of $u$ fulfills
  $Du = \sigma$. As $\sigma\in \wilde H^{\alpha/2-1}(I)$, we conclude that $u\in H^{\alpha/2}(I)$.
  In a second step, using functions $\tau\in C^\infty(\el)$ for all $\el\in\TT$ in~\eqref{eq:bf:a} and
  integrating by parts shows that $\wat u = u$ at inner nodes as well as
  $u(a)=u(b)=0$.
  Hence $u\in \wilde H^{\alpha/2}(I)$.
  We plug in $\sigma = Du$ in~\eqref{eq:bf:b} and obtain the variational formulation
  \begin{align*}
    \vdual{D^{\alpha-2}Du}{Dv} + \vdual{bDu}{v} + \vdual{cu}{v} = 0
    \quad\text{ for all } v\in \wilde H^{\alpha/2}(I).
  \end{align*}
  According to~\cite[Section~3]{ErvinR_NumPDE_06}, the bilinear form on the left-hand side
  of this formulation is elliptic on $\wilde H^{\alpha/2}(I)$. We conclude that $u=0$ and
  hence $\sigma=0$. Then, $\wat u = 0$ and $\wat\sigma=0$ follow immediately.
\end{proof}
%---------------------------------------------------------------------------------------------------------------
\subsection{Analysis of the adjoint problem}
%---------------------------------------------------------------------------------------------------------------
\begin{lemma}\label{lemma:adjoint:inhom}
  For $F\in L_2(I)$ and $G\in H^{1-\alpha/2}(I)$, there exists a solution
  $\tau\in H^1(I)$, $v\in \wilde H^{\alpha/2}(I)$ of
  \begin{align}\label{lemma:adjoint:inhom:eq}
    \begin{split}
      D\tau + cv &= F\\
      \tau + D^{(\alpha-2) \star} Dv + bv &= G
    \end{split}
  \end{align}
  such that
  \begin{align}\label{lemma:adjoint:inhom:stab}
    \norm{v}{H^{\alpha/2}(I)} + \norm{\tau}{H^1(I)} \lesssim \norm{F}{L_2(I)} + \norm{G}{H^{1-\alpha/2}(I)}.
  \end{align}
\end{lemma}
\begin{proof}
  Consider the variational formulation to find $v\in \wilde H^{\alpha/2}(I)$ such that
  \begin{align*}
    \vdual{Dv}{D^{\alpha-2}D\phi} + \vdual{bv}{D\phi} + \vdual{cv}{\phi} = \vdual{F}{\phi} - \vdual{DG}{\phi}
    \quad\text{ for all }\phi\in\wilde H^{\alpha/2}(I).
  \end{align*}
  According to~\cite[Section.~3]{ErvinR_NumPDE_06}, the bilinear form of this formulation is elliptic
  on $\wilde H^{\alpha/2}(I)$. The linear functional on the right-hand side is bounded
  in $\wilde H^{\alpha/2}(I)$
  with constant $\norm{F}{L_2(I)} + \norm{G}{H^{1-\alpha/2}(I)}$.
  Hence, there exists a unique solution $v\in \wilde H^{\alpha/2}(I)$ which satisfies
  \begin{align*}
    \norm{v}{\wilde H^{\alpha/2}(I)} \lesssim \norm{F}{L_2(I)} + \norm{G}{H^{1-\alpha/2}(I)}
  \end{align*}
  Now define $\tau := - D^{(\alpha-2) \star} Dv - bv + G$. A priori, $\tau\in H^{1-\alpha/2}(I)$,
  but the definition of $v$ shows
  \begin{align*}
    \vdual{\tau}{D\phi} = \vdual{cv-F}{\phi}\quad\text{ for all }\phi\in C_0^\infty(I).
  \end{align*}
  Hence, $\tau\in H^1(I)$ and $D\tau = F-cv$. The bounds on $\tau$ follow immediately.
\end{proof}
%---------------------------------------------------------------------------------------------------------------
\begin{lemma}\label{lemma:adjoint:hom}
  Suppose that $\tau\in H^1(\TT)$ and $v\in H^{\alpha/2}(\TT)$ fulfill
  \begin{subequations}\label{lemma:adjoint:hom:eq}
    \begin{align}
      D_\TT\tau + cv &= 0\label{lemma:adjoint:hom:eq:a}\\
      \tau + D^{(\alpha-2) \star} D_\TT v + bv &= 0,\label{lemma:adjoint:hom:eq:b}
    \end{align}
  \end{subequations}
 on $I$. Then
  \begin{align*}
    \norm{\tau}{H^1(\TT)} + \norm{v}{H^{\alpha/2}(\TT)} \lesssim 
    N^{3/2}\left( \abs{\jump{v}} + \abs{\jump{\tau}} \right).
  \end{align*}
\end{lemma}
\begin{proof}
  We proceed in three steps.\\

  {\bf Step 1:} Let $\psi\in \wilde H^{\alpha/2}(I)$ be the unique variational
  solution of $-DD^{\alpha-2}D\psi+bD\psi +c\psi= -v$,
  cf.~\cite[Section~3]{ErvinR_NumPDE_06}, i.e.,
  \begin{align*}
    \vdual{D^{\alpha-2}D\psi}{D\phi} + \vdual{bD\psi}{\phi} +
    \vdual{c\psi}{\phi} = -\vdual{v}{\phi}\quad
    \text{ for all } \phi\in\wilde H^{\alpha/2}(I),
  \end{align*}
  so that $\norm{\psi}{\wilde H^{\alpha/2}(I)} \lesssim
  \norm{v}{L_2(I)}$. Due to Lemma~\ref{lemma:extension}
  and $\alpha-2\leq \alpha/2-1$, it holds
  \begin{align}\label{lemma:adjoint:hom:eq3}
    \norm{D^{\alpha-2}D\psi}{L_2(I)} &\lesssim \norm{D\psi}{\wilde H^{\alpha-2}(I)}
    \lesssim \norm{D\psi}{\wilde H^{\alpha/2-1}}
%    \lesssim \norm{D\psi}{H^{\alpha-2(I)}}
%    \lesssim \norm{\psi}{H^{\alpha-1}(I)} \lesssim
    \lesssim \norm{\psi}{\wilde H^{\alpha/2}(I)} \lesssim \norm{v}{L_2(I)}.
  \end{align}
  The equation solved by $\psi$ implies that the distributional derivative of
  $D^{\alpha-2}D\psi$ is given by
  $DD^{\alpha-2}D\psi = bD\psi+c\psi +v\in H^{\alpha/2-1}(I)$, such that
  $D^{\alpha-2}D\psi\in H^{\alpha/2}(I)$. Using Lemma~\ref{lemma:poincare},
  we see
  \begin{align}\label{lemma:adjoint:hom:eq4}
    \begin{split}
      \snorm{D^{\alpha-2}D\psi}{H^{\alpha/2}(I)} &\lesssim
      \norm{DD^{\alpha-2}D\psi}{H^{\alpha/2-1}(I)}
      = \norm{bD\psi + c\psi + v}{H^{\alpha/2-1}(I)}\\
      &\lesssim \norm{\psi}{\wilde H^{\alpha/2}(I)} + \norm{v}{L_2(I)}
      \lesssim \norm{v}{L_2(I)}.
    \end{split}
  \end{align}
  We may also integrate by parts and use~\eqref{lemma:adjoint:hom:eq} to obtain
  
  \begin{align*}
    \vdual{v}{v} &= -\vdual{D^{\alpha-2}D\psi}{D_\TT v} - \vdual{bD\psi}{v} -
    \vdual{c\psi}{v} + \dual{D^{\alpha-2}D\psi}{\jump{v}}\\
    &=\vdual{D\psi}{-D^{(\alpha-2)\star}D_\TT v - bv - \tau} +
    \dual{\psi}{\jump{\tau}} + \dual{D^{\alpha-2}D\psi}{\jump{v}}\\
    &= \dual{\psi}{\jump{\tau}} + \dual{D^{\alpha-2}D\psi}{\jump{v}}.
  \end{align*}
  
  Lemma~\ref{lemma:scaling}, eq.~\eqref{lemma:scaling:3},
  estimates~\eqref{lemma:adjoint:hom:eq3},~\eqref{lemma:adjoint:hom:eq4}, and stability of $\psi$
  yield
  \begin{align}\label{lemma:adjoint:hom:eq1}
    \norm{v}{L_2(I)}^2 \lesssim N^{1/2} \left( 
      \abs{\jump{\tau}} + \abs{\jump{v}} \right) \norm{v}{H^{\alpha/2}(\TT)}.
  \end{align}
  {\bf Step 2:}
  Piecewise integration by parts shows
  \begin{align*}
    \vdual{D_\TT(bv)}{v} &= \vdual{v Db}{v} + \vdual{bD_\TT v}{v}\\
    &= \vdual{v Db}{v} - \vdual{v}{D_\TT(bv)} - \dual{\jump{v}}{\avg{bv}} -
    \dual{\avg{v}}{\jump{bv}},
  \end{align*}
  which gives
  \begin{align}\label{lemma:adjoint:hom:eq8}
    \vdual{D_\TT(bv)}{v} &= \vdual{v Db/2}{v} - 1/2\dual{\jump{v}}{\avg{bv}} -
    1/2\dual{\avg{v}}{\jump{bv}}.
  \end{align}
  Now we multiply~\eqref{lemma:adjoint:hom:eq:a} with $v$ and
    insert~\eqref{lemma:adjoint:hom:eq:b} as well as~\eqref{lemma:adjoint:hom:eq8}.
  Then, as $D^{(\alpha-2)\star} D_\TT v\in
  H^{\alpha/2}(\TT)$ by~\eqref{lemma:adjoint:hom:eq:b}, integration by parts
  gives
  \begin{align*}
    0 &= \vdual{D_\TT v}{D^{\alpha-2}D_\TT v} + \vdual{v(c-Db/2)}{v}\\
    &\quad+ \dual{\jump{D^{(\alpha-2)\star}D_\TT v}}{\avg{v}}
    + \dual{\avg{D^{(\alpha-2)\star}D_\TT v}}{\jump{v}}
    + 1/2 \dual{\jump{v}}{\avg{bv}} + 1/2 \dual{\avg{v}}{\jump{bv}}
  \end{align*}
  As $c-Db/2\geq0$ and $D^{\alpha-2}$ is elliptic in $H^{\alpha/2-1}(\TT)$ due to
  Lemma~\ref{lemma:rli:elliptic}, we obtain with the triangle inequality
%  with Lemma~\ref{lemma:poincare} and the trace inequality
  \begin{align*}
    \begin{split}
%      \snorm{v}{H^{\alpha/2}(\TT)}^2 &\lesssim
      \norm{D_\TT v}{H^{\alpha/2-1}(\TT)}^2 
      \lesssim
      \abs{\dual{\jump{\tau}}{\avg{v}}} + \abs{\dual{\avg{\tau}}{\jump{v}}}
      + \abs{\dual{\jump{v}}{\avg{bv}}} + \abs{\dual{\jump{bv}}{\avg{v}}}.
%      &\lesssim N^{1/2}\abs{\jump{v}}\cdot\norm{v}{H^{\alpha/2}(\TT)}
    \end{split}
  \end{align*}
  All terms on the right-hand side of this inequality are
  treated with Lemma~\ref{lemma:scaling}, eq.~\eqref{lemma:scaling:3}.
  For the second term, we additionally use Lemma~\ref{lemma:poincare:pw}
  and~\eqref{lemma:adjoint:hom:eq:a} and get
  \begin{align*}
    \abs{\dual{\avg{\tau}}{\jump{v}}} \lesssim N^{1/2}\norm{\tau}{H^1(\TT)}
    \cdot\abs{\jump{v}}
    &\lesssim N^{1/2}\norm{v}{L_2(I)} \abs{\jump{v}} + N
    \abs{\jump{\tau}} \cdot \abs{\jump{v}}\\
    &\lesssim N^{1/2}\norm{v}{L_2(I)} \abs{\jump{v}}
    + N^{3/2} \norm{v}{H^{\alpha/2}(\TT)} \cdot\abs{\jump{\tau}}.
  \end{align*}
  We conclude that
  \begin{align}\label{lemma:adjoint:hom:eq2}
    \snorm{v}{H^{\alpha/2}(\TT)}^2 \lesssim
    \norm{D_\TT v}{H^{\alpha/2-1}(\TT)}^2 
    \lesssim
    N^{3/2}\left( \abs{\jump{\tau}} + \abs{\jump{v}}
    \right)\norm{v}{H^{\alpha/2}(\TT)},
  \end{align}
  where we have used Lemma~\ref{lemma:poincare} for the first estimate.
  Adding~\eqref{lemma:adjoint:hom:eq1} and~\eqref{lemma:adjoint:hom:eq2} and
  dividing by $\norm{v}{H^{\alpha/2}(\TT)}$ gives
  \begin{align}\label{lemma:adjoint:hom:eq7}
    \norm{v}{H^{\alpha/2}(\TT)} \lesssim N^{3/2} \left( 
      \abs{\jump{\tau}} + \abs{\jump{v}} \right).
  \end{align}
  {\bf Step 3:} It remains to show the bound for $\tau$.
  As $\tau\in L_2(I)$, we can write $\tau=D\psi+t$
  with $\psi\in \wilde H^1(I)$ and $t\in\R$ such that
  $\norm{\psi}{H^1(I)} + \abs{t}\lesssim\norm{\tau}{L_2(I)}$.
  Integration by parts,
  identities~\eqref{lemma:adjoint:hom:eq}, Lemma~\ref{lemma:scaling} eq.~\eqref{lemma:scaling:3},
  and Cauchy-Schwarz show
  \begin{align}\label{lemma:adjoint:hom:eq5}
    \begin{split}
      \vdual{\tau}{\tau} &= \vdual{cv}{\psi} + \dual{\jump{\tau}}{\psi}
      - \vdual{D^{(\alpha-2)\star}D_\TT v+bv}{t}\\
      &\lesssim 
      \left( \norm{v}{L_2(I)} + N^{1/2}\abs{\jump{\tau}}
      + \vdual{D^{(\alpha-2)\star}D_\TT v}{1}\right)\norm{\tau}{L_2(I)}.
    \end{split}
  \end{align}
  For the last term, we use $(D^{\alpha-2}1)(x) = x^{2-\alpha}/\Gamma(2-\alpha+1)$,
  cf.~\cite[Section~2.5]{SamkoKM_Fractional},
  and integration by parts to compute
  \begin{align}\label{lemma:adjoint:hom:eq6}
    \begin{split}
      \vdual{D^{(\alpha-2)\star}D_\TT v}{1}
      &= -\vdual{v}{x^{1-\alpha}} +
      \frac{\dual{\jump{v}}{x^{2-\alpha}}}{\Gamma(2-\alpha+1)}\\
      &\lesssim \norm{v}{H^{\alpha/2}(\TT)}
      + N^{1/2}\abs{\jump{v}}.
    \end{split}
  \end{align}
  Here, the last estimate follows by direct computation.
  Combining the estimates~\eqref{lemma:adjoint:hom:eq5},~\eqref{lemma:adjoint:hom:eq6},
  and~\eqref{lemma:adjoint:hom:eq7}, we obtain
  \begin{align*}
    \norm{\tau}{L_2(I)} \lesssim
    N^{3/2} \left( \abs{\jump{\tau}} + \abs{\jump{v}} \right).
  \end{align*}
  An estimate for $D_\TT \tau$ is obtained from~\eqref{lemma:adjoint:hom:eq:a}
  and~\eqref{lemma:adjoint:hom:eq7}. This concludes the proof.
%  {\bf Step 2:}
%  As $\tau\in L_2(I)$, we can write $\tau = D\psi + t$ with $\psi\in \wilde H^1(I)$ and $t\in\R$. Then,
%  integration by parts and~\eqref{lemma:adjoint:hom:eq} show
%  \begin{align*}
%    \vdual{\tau}{\tau} = \vdual{cv}{\psi} + \dual{\jump{\tau}}{\psi} - \vdual{D^{(\alpha-2) \star} D_\TT v + bv}{t}.
%  \end{align*}
%  Hence
%  \begin{align*}
%    \norm{\tau}{L_2(I)}^2 \lesssim \norm{v}{L_2(I)}\cdot{\psi}{L_2(I)} + N^{1/2}
%  \end{align*}
%  Now
%  \begin{align*}
%    
%  \end{align*}
\end{proof}
%============================================================================================

%===============================================================================================================
\section{Numerical Examples}\label{section:numerics}
%===============================================================================================================

\subsection{Discretization and approximated optimal test functions}
Let us briefly fix some notation: We consider the discrete subspace
\begin{align*}
  U_{\hp}(\TT) := U^p(\TT)\times U^q(\TT) \times \R^{N+1} \times \R^{N-1} \subset U_\alpha,
\end{align*}
where 
\begin{align*}
  U^p(\TT) := \{ v\in L_2(I) \,:\, v|_T \text{ is polynomial of degree at most } p \;\forall T\in\TT \}
\end{align*}
is the space of $\TT$-elementwise polynomials of degree $p\in \N_0$. 
%Given the affine map $\gamma_T : [0,1] \to [a,b]$ with $\gamma_T(t) = a + (b-1)t$, we define the functions $\eta_{T,j}$
%by setting
%\begin{align*}
%  (\eta_{T,j}\circ \gamma_T^{-1})(t) := t^j \quad\text{and}\quad \eta_{T,j} |_{T'} = 0 \quad\text{for } T\neq T'\in\TT.
%\end{align*}
%Our basis of $U_h^p(\TT)$ is then given by
%\begin{align*}
%  \{ \eta_{T,j} \,:\, T\in \TT, j=0,\dots,p \}.
%\end{align*}
%For $\R^{N+1}$ resp. $\R^{N-1}$ we use the canonical Euclidean basis vectors denoted by $e_j$. 
%Our basis for $U_h(\TT)$ is given by the functions $\{(\eta_{T,i},0,0,0),
%(0,\eta_{T,j},0,0),(0,0,e_k,0),(0,0,0,e_\ell)\}$ for $T\in \TT, i=0,\dots,p, j=0,\dots,q, k=1,\dots,N+1,
%\ell=1,\dots,N-1$.
Note that $\dim(U_{\hp}(\TT))=(p+q+4)N$.
Given a basis $\left\{ \uu_j \mid j=1, \dots, \dim(U_\hp(\TT)) \right\}$
of $U_{\hp}(\TT)$,
the optimal test functions $\TtoT(\uu_j)\in V_\alpha$
($j=1,\dots,\dim(U_\hp(\TT))$ are computed by solving the problems
\begin{align}\label{eq:def:optimalTestFun}
  \dual{\TtoT(\uu_j)}{\vv}_{V_\alpha} = b(\uu_j,\vv) \quad\text{for all }
  \vv\in V_\alpha = H^1(\TT)\times H^{\alpha/2}(\TT).
\end{align}
For $\vv = (\tau,v),\ww = (\rho,w)\in V_\alpha$ the $V_\alpha$-inner product is given by
\begin{align*}
  \dual{\vv}{\ww}_{V_\alpha} = (\tau,\rho)_I + (D_\TT \tau,D_\TT \rho)_I + (v,w)_I + \sum_{T\in\TT} \int_T\int_T
  \frac{(v(x)-v(y))(w(x)-w(y))}{|x-y|^{1+\alpha}} \,dy\,dx,
\end{align*}
which induces our chosen local norm
$\norm{\vv}{V_\alpha}^2 = \norm{\tau}{H^1(\TT)}^2 + \norm{v}{H^{\alpha/2}(\TT)}^2$ on $V_\alpha$.
Since the definition of the optimal test functions~\eqref{eq:def:optimalTestFun} involves
the infinite-dimensional space
$V_\alpha$, we approximate $\TtoT_\alpha(\uu_j)\in V_\alpha$ by $\TtoT_{\alpha,h}(\uu_j) \in
V_\hp(\TT) := U^m(\TT)\times U^n(\TT)$ with $m,n\in\N_0$,
i.e., instead of~\eqref{eq:def:optimalTestFun} we solve for $j=1,\dots,\dim(U_\hp(\TT))$ the problem
\begin{align}
  \dual{\TtoT_{\alpha,h}(\uu_j)}{\vv_k}_{V_\alpha} = b(\uu_j,\vv_k) \quad
  k=1,\dots,\dim(V_\hp(\TT)).
\end{align}
%The set of basis functions for $V_\hp(\TT)$ is given by
%\begin{align*}
%  \{(\eta_{T,i},0),(0,\eta_{T,j}) \,:\, T\in\TT, i=0,\dots,m, j = 0,\dots,n\}.
%\end{align*}
The inner product $\dual{\vv}{\ww}_{V_\alpha}$ is computed analytically for
functions $\vv,\ww\in V_\hp(\TT)$.
It is seen immediately that choosing $m$ and $n$ too small in comparison with $p$ and $q$
  leads to a system which is not well posed. This question is investigated in~\cite{gq14}.
  The authors show that in the case of the Poisson equation in $\mathbb R^d$ and $p=q$,
  using polynomial degrees $n=m$ which are higher than $p+d$ is sufficient in order to obtain
  well-posedness and best approximation results.

Altogether we have to assemble the matrices
$\BLFmat := (\BLFmat_{kj})$ and $\TtoTmat:= (\TtoTmat_{k\ell})$ with
\begin{align*}
  \BLFmat_{kj} := b(\uu_j,\vv_k)\quad\text{ and }\quad
  \TtoTmat_{k\ell} := \dual{\vv_\ell}{\vv_k}_{V_\alpha},
\end{align*}
where $\uu_j$ and $\vv_k$, $j=1,\dots,\dim(U_\hp(\TT))$, $k = 1,\dots,\dim(V_\hp(\TT))$,
are the basis functions described above.
Note that $\TtoTmat$ has a sparse structure, whereas $\BLFmat$ contains a dense block
corresponding to the discretization of the fractional integral operator.
With the definition of the right-hand side vector
\begin{align*}
  \ff_j := \ell(\vv_j) \quad\text{for all } j=1,\dots,\dim(V_\hp(\TT))
\end{align*}
the computation of the DPG solution~\eqref{thm:dpg:eq:discrete} consists in solving
the linear system
\begin{align}
  \BLFmat^T \TtoTmat^{-1} \BLFmat \xx = \BLFmat^T \TtoTmat^{-1} \ff.
\end{align}

An advantage of the DPG method is that, by design, we can evaluate the error in the energy norm.
We define the local contributions of the error in the energy norm
on an element $\el\in\TT$, $\est(T)$, as
\begin{align}
  \est(T)^2 := \sum_{\substack{\{j\,:\, \vv_j|_{T'} = 0 \text{ for } T'\neq T \} }}
  (\ff-\BLFmat\xx)_j \left( \TtoTmat^{-1}(\ff-\BLFmat\xx) \right)_j
\end{align}
Then, with $r_h \in V_\hp$ denoting the element corresponding to the vector $\TtoTmat^{-1}(\ff-\BLFmat\xx)$ it holds
\begin{align}
  \est^2 := \sum_{T\in\TT} \est(T)^2 = (\ff-\BLFmat\xx)^T \left( \TtoTmat^{-1}(\ff-\BLFmat\xx) \right) = \norm{r_h}{V_\alpha}^2.
\end{align}

Let us discuss the convergence rates we can expect.
Due to standard approximation results of the $L_2$-orthogonal projection
$\pi_p : L_2(I)\rightarrow U^p(\TT)$ we have
\begin{align*}
  \inf_{\sigma_\hp'\in U^p(\TT)} N^{1-\alpha/2}\norm{\sigma-\sigma_\hp'}{H^{\alpha/2-1}(I)}
  &\leq N^{1-\alpha/2}\norm{\sigma-\pi_p \sigma}{H^{\alpha/2-1}(I)}\\
  &\lesssim \norm{\sigma-\pi_p \sigma}{L_2(I)}
  \lesssim N^{-\min\left( p+1,s \right)} \norm{\sigma}{H^s(I)}
\end{align*}
and
\begin{align*}
  \inf_{u_\hp'\in U^q(\TT)} \norm{u-u_\hp'}{L_2(I)} &\lesssim N^{-\min\left( q+1,r \right)} \norm{u}{H^r(I)}.
\end{align*}
  According to Theorem~\ref{thm:main}, this yields
\begin{align}\label{eq:rates}
  \norm{\uu-\uu_\hp}{U_\alpha} \lesssim \est \lesssim N^{-\min\left( q+1, p+1, r, s \right)}
  \left( \norm{\sigma}{H^s(I)} + \norm{u}{H^r(I)} \right).
\end{align}
Here, the fact that $\est$ can be included in this estimate in this way follows from Theorem~\ref{thm:dpg}.
For the numerical examples where the exact solution $\uu = (\sigma,u,\widehat\sigma,\widehat u)$ is known,
we can compute the exact error $\norm{\uu}{U_\alpha}$. For this we define the quantities
\begin{align*}
  \err(u_h) &:= \norm{u-u_h}{L_2(I)}, \\
  \err(\sigma_h) &:= N^{1-\alpha/2}\norm{h^{\alpha/2-1}(\sigma-\sigma_h)}{L_2(I)}, \\
  \err(\widehat u_h) &:= N^{-1/2} |\widehat u - \widehat u_h|, \\
  \err(\widehat \sigma_h) &:= N^{-1/2} |\widehat \sigma - \widehat\sigma_h|.
\end{align*}
Here, $\widehat u$ are the evaluations of the function $u$ at the interior nodes
(i.e. without the endpoints of the interval $I=(0,1)$) of the mesh $\TT$ and
$\widehat\sigma$ are the evaluations of $D^{2-\alpha}Du$ at the nodes of $\TT$.
We emphasize that the norms $\err(\sigma_h)$, $\err(\widehat u_h)$, and $\widehat(\widehat\sigma_h)$
to measure the error of approximations to $\sigma$, $\wat u$, and $\wat\sigma$ are
\textit{stronger} than those contained in the norm $\norm{\uu}{U_\alpha}$ on the
left-hand side of~\eqref{eq:rates}. 
However, the experiments show that we have optimal convergence rates also in these stronger norms.
We emphasize that we even have the rigorous error bound
\begin{align*}
  \norm{\uu-\uu_\hp}{U_\alpha}^2 \lesssim \est^2 \lesssim \err(\sigma_h)^2 + \err(u_h)^2.
\end{align*}

\subsection{Example~1}\label{sec:example1}

\begin{figure}[htb]
  \centering
  \psfrag{nE}[c][c]{\tiny number of elements}
  \psfrag{error}[c][c]{\tiny errors }
  \psfrag{N1}[c][c]{\tiny $\OO(N^{-1})$}
  \psfrag{N2}[c][c]{\tiny $\OO(N^{-2})$}
  \psfrag{N25}[c][c][1][-30]{\tiny $\OO(N^{-5/2})$}
  \psfrag{p0q0n2m2}[c][c]{\tiny $p=0, q=0, m=2, n=2$}
  \psfrag{p1q0n2m2}[c][c]{\tiny $p=1, q=0, m=2, n=2$}
  \psfrag{p0q1n2m2}[c][c]{\tiny $p=0, q=1, m=2, n=2$}
  \psfrag{p1q1n3m3}[c][c]{\tiny $p=1, q=1, m=3, n=3$}

  \psfrag{errEst}{\tiny $\est$}
  \psfrag{errU}{\tiny $\err(u_h)$}
  \psfrag{errUhat}{\tiny $\err(\widehat u_h)$}
  \psfrag{errSigmaHat}{\tiny $\err(\widehat \sigma_h)$}
  \psfrag{errSigma}{\tiny $\err(\sigma_h)$}

  \includegraphics[width=0.49\textwidth]{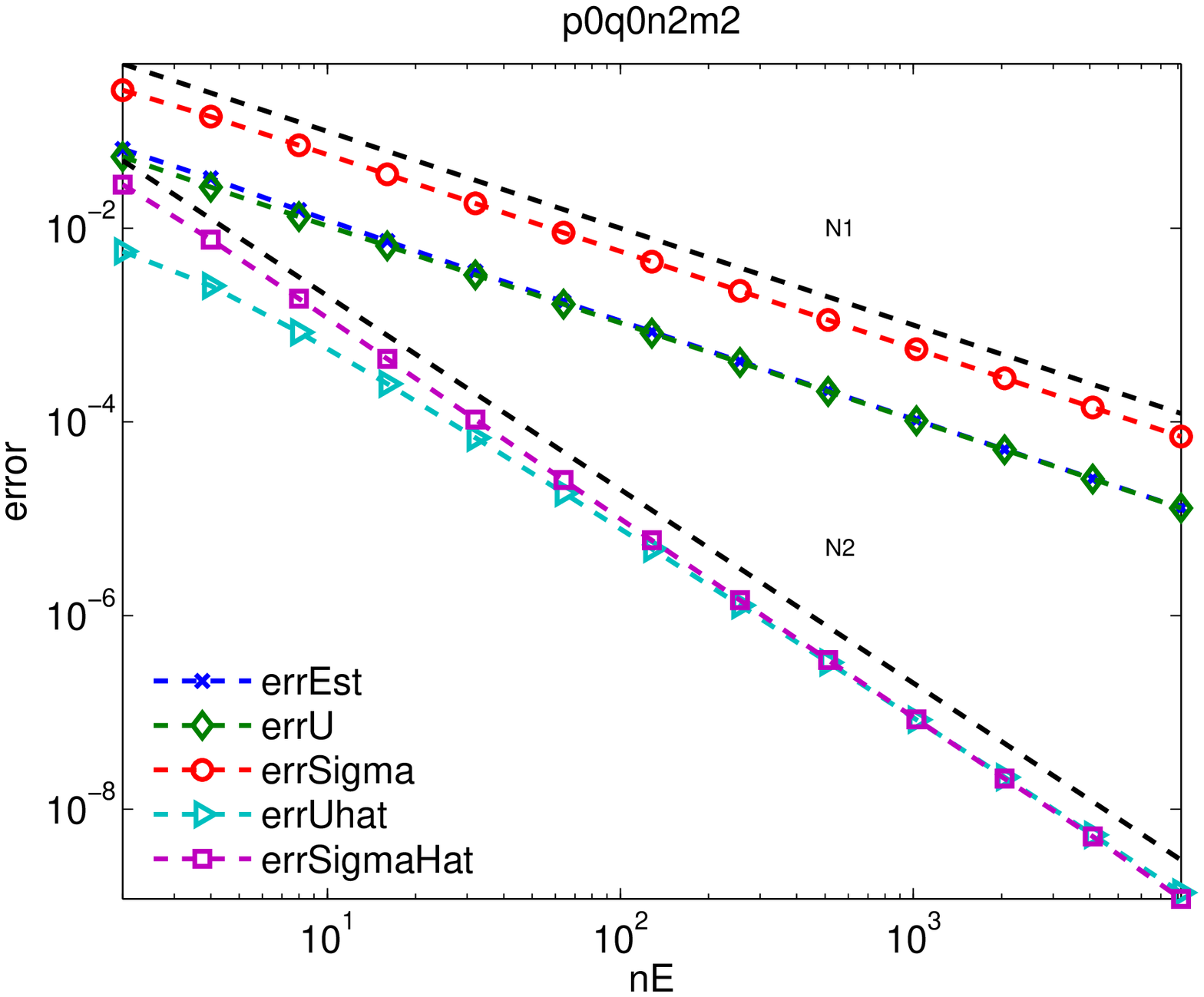}
  \includegraphics[width=0.49\textwidth]{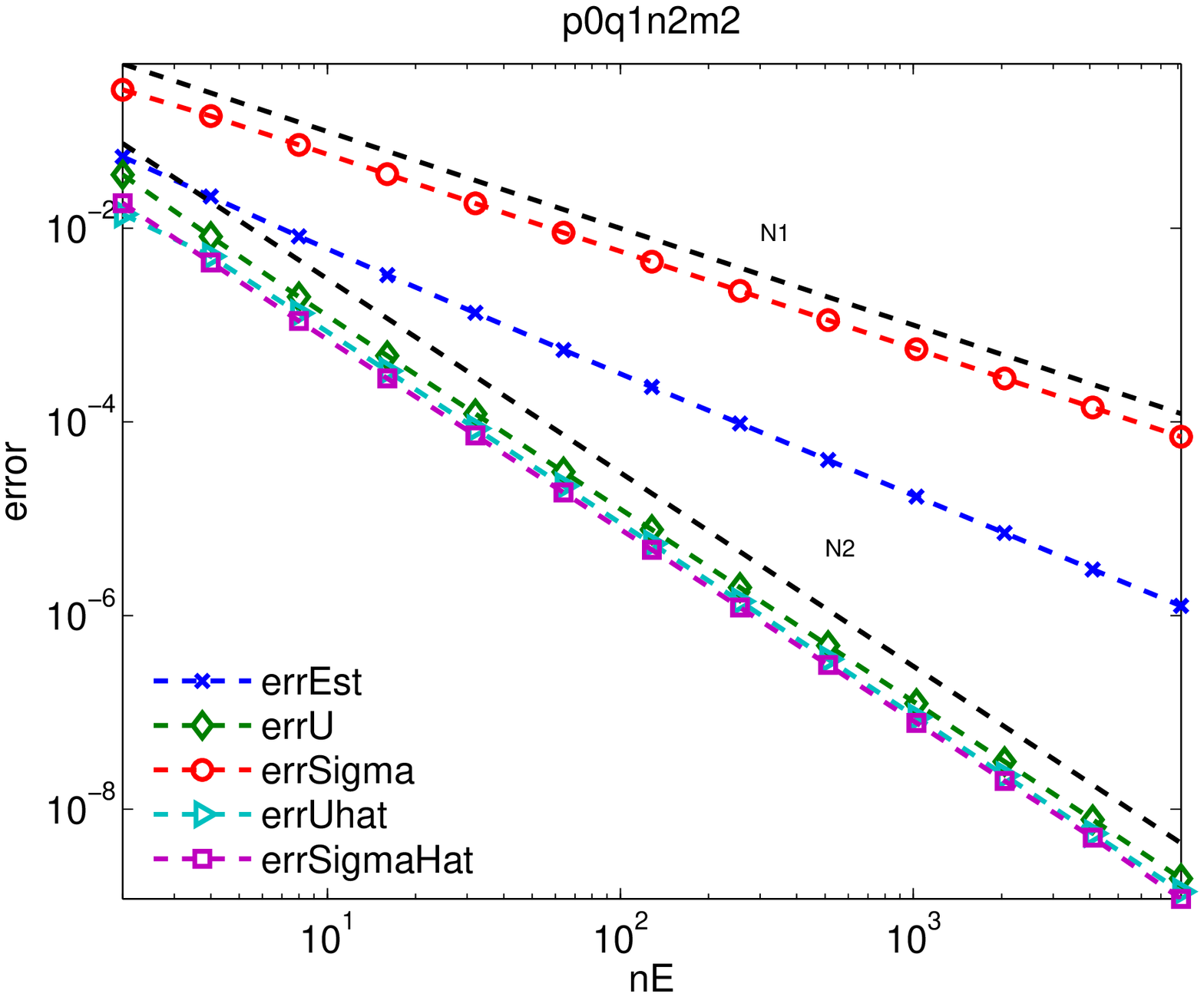}
  \includegraphics[width=0.49\textwidth]{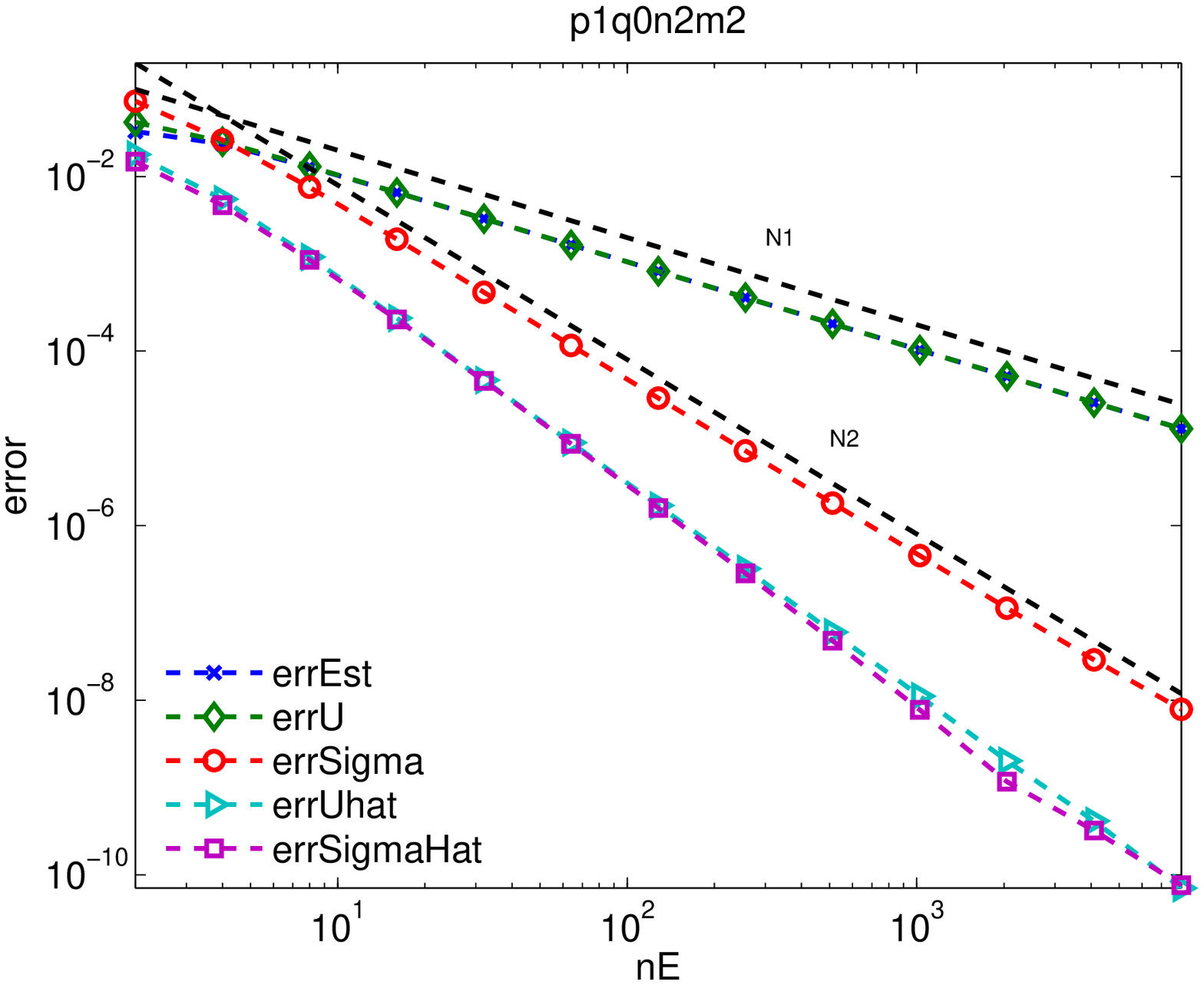}
  \includegraphics[width=0.49\textwidth]{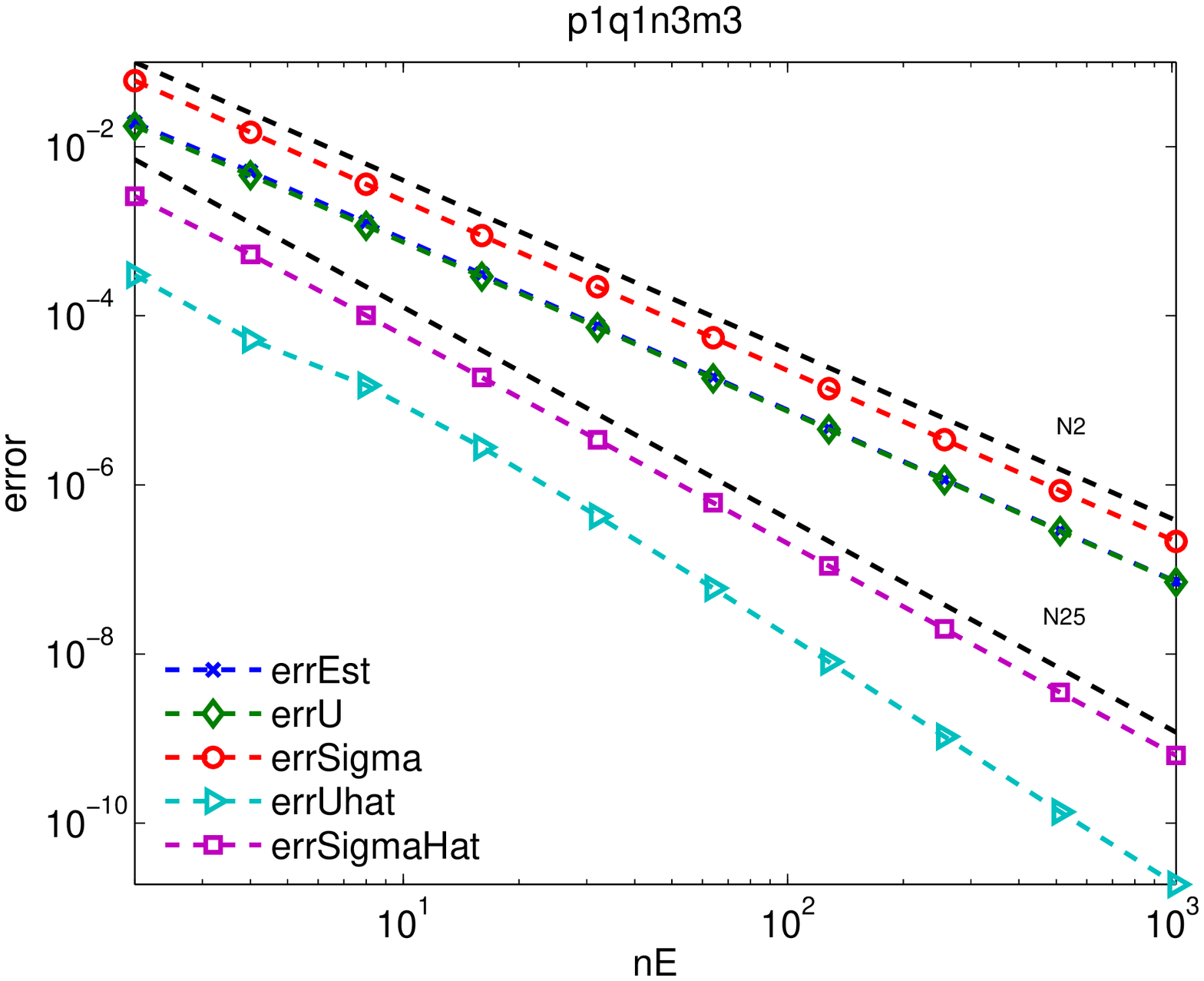}

  \caption{Experimental convergence rates for Example 1 from Section~\ref{sec:example1}.
    Uniform mesh refinement is used througout.}
  \label{fig:example1}
\end{figure}

We consider the following example, see also~\cite[Section~5,
Example~2]{ErvinR_NumPDE_06}:
Let $I = (0,1)$, $\alpha = 3/2$, $b(x):= 1/2$, $c(x):=1/2$ for $x\in I$ and prescribe the exact solution $u(x) = x^2-x^3$.
Then, the right-hand side is given by
\begin{align*}
  f(x) = -2\frac{\Gamma(2)}{\Gamma(3-\alpha)} x^{2-\alpha} +3\frac{\Gamma(3)}{\Gamma(4-\alpha)} x^{3-\alpha}
  -\frac12 x^3 - x^2 + x.
\end{align*}
Furthermore, straightforward calculations show
\begin{align*}
  \sigma(x) &= Du = 2x-3x^2, \\
  D^{2-\alpha}\sigma(x) &=  D^{2-\alpha}D u = 2\frac{\Gamma(2)}{\Gamma(4-\alpha)} x^{3-\alpha} -
  3\frac{\Gamma(3)}{\Gamma(5-\alpha)} x^{4-\alpha}.
\end{align*}
%Note that $\widehat u \in\R^{N-1}$ are the point evaluations of $u$ at the nodes of $\TT$ (without endpoints $0,1$) and
%$\widehat \sigma \in\R^{N+1}$ are the
%point evaluations of $D^{2-\alpha}\sigma$.
We consider uniform meshes on $I$ with mesh-size $h = 1/N$ and $N = \#\TT$.
Figure~\ref{fig:example1} shows results for different values of $p,q,m,n$.
As $u$ and $\sigma$ are both smooth, we expect from~\eqref{eq:rates} that
$\est = \OO(N^{-\min\left( p+1,q+1 \right)})$, and the numerical experiments
reflect this expectation. We even see in the experiments the simultaneous approximation orders
$\err(u_h)=\OO(N^{-(q+1)})$ and $\err(\sigma_h)=\OO(N^{-(p+1)})$.
The trace errors $\err(\widehat u_h)$ and $\err(\widehat\sigma_h)$ show higher convergence rates in all cases.
In the case $p=0, q=1, m=2,n=2$ (upper right plot),
$\est$ converges slightly faster than $\err(\sigma_h)$ but slower than $\err(u_h)$.

\subsection{Example~2}\label{sec:example2}

\begin{figure}[htb]
  \centering
  \psfrag{nE}[c][c]{\tiny number of elements}
  \psfrag{error}[c][c]{\tiny errors}
  \psfrag{N1}[c][c]{\tiny $\OO(N^{-1})$}
  \psfrag{N2}[c][c]{\tiny $\OO(N^{-2})$}
  \psfrag{N3}[c][c]{\tiny $\OO(N^{-3})$}
  \psfrag{N4}[c][c]{\tiny $\OO(N^{-4})$}
  \psfrag{Nmlp110}[c][c]{\tiny $\OO(N^{-\lambda+1/10})$}
  \psfrag{Nmlm110}[c][c]{\tiny $\OO(N^{-\lambda-1/10})$}
  \psfrag{N12l}[c][c]{\tiny $\OO(N^{1/2-\lambda})$}
  \psfrag{N810}[c][c]{\tiny $\OO(N^{-8/10})$}
  \psfrag{N25}[c][c][1][-30]{\tiny $\OO(N^{-5/2})$}

  \psfrag{p0q0n2m2}[c][c]{\tiny $p=0, q=0, m=2, n=2, \theta =1$}
  \psfrag{p0q0n2m2Adap}[c][c]{\tiny $p=0, q=0, m=2, n=2, \theta =0.4$}
  \psfrag{p1q1n3m3}[c][c]{\tiny $p=1, q=1, m=3, n=3, \theta=0.4$}
  \psfrag{p2q2n4m4}[c][c]{\tiny $p=2, q=2, m=4, n=4, \theta=0.4$}
 
  \psfrag{errEst}{\tiny $\est$}
  \psfrag{errU}{\tiny $\err(u_h)$}
  \psfrag{errUhat}{\tiny $\err(\widehat u_h)$}
  \psfrag{errSigmaHat}{\tiny $\err(\widehat \sigma_h)$}
  \psfrag{errSigma}{\tiny $\err(\sigma_h)$}

  \includegraphics[width=0.49\textwidth]{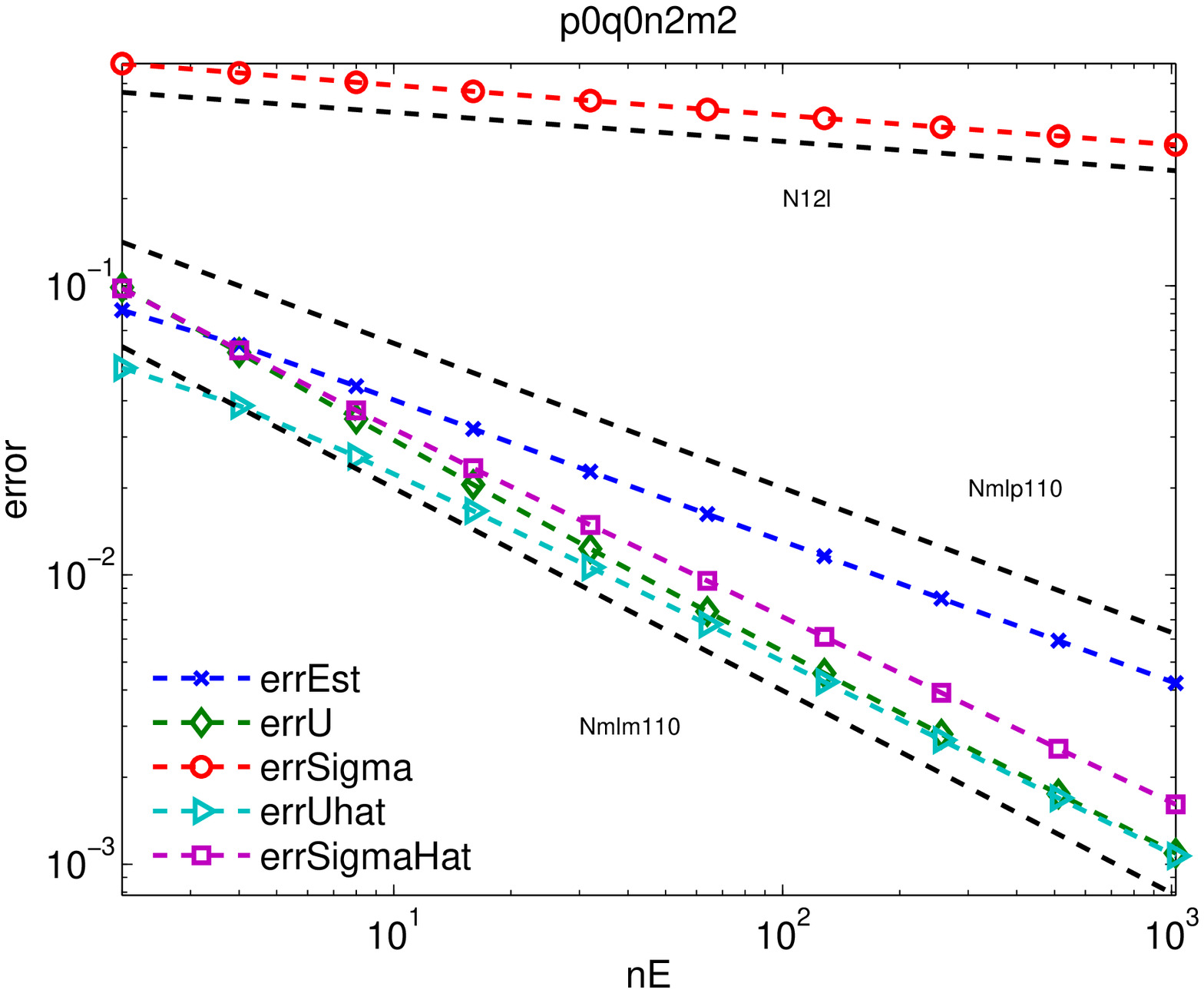}
  \includegraphics[width=0.49\textwidth]{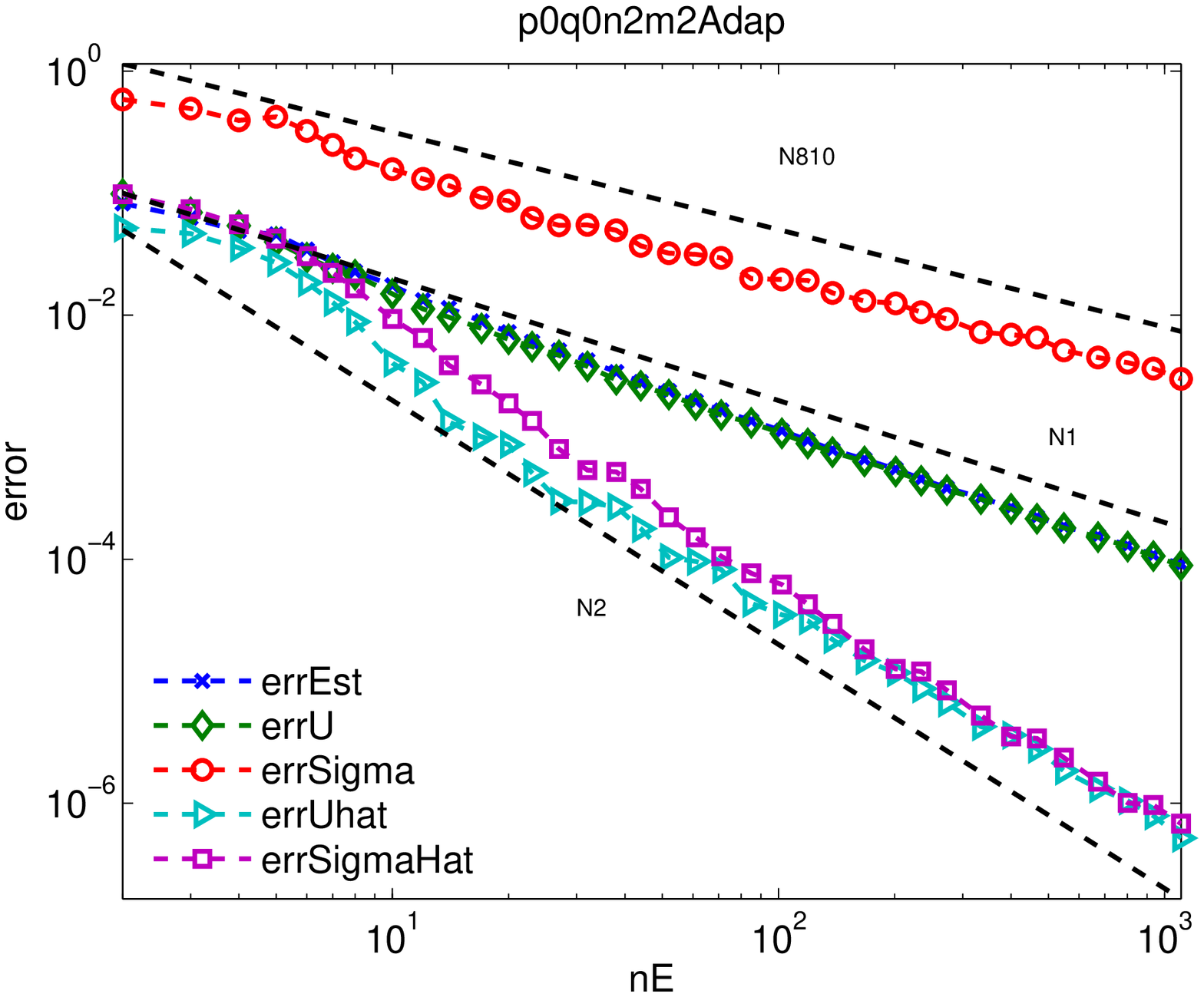}
  \includegraphics[width=0.49\textwidth]{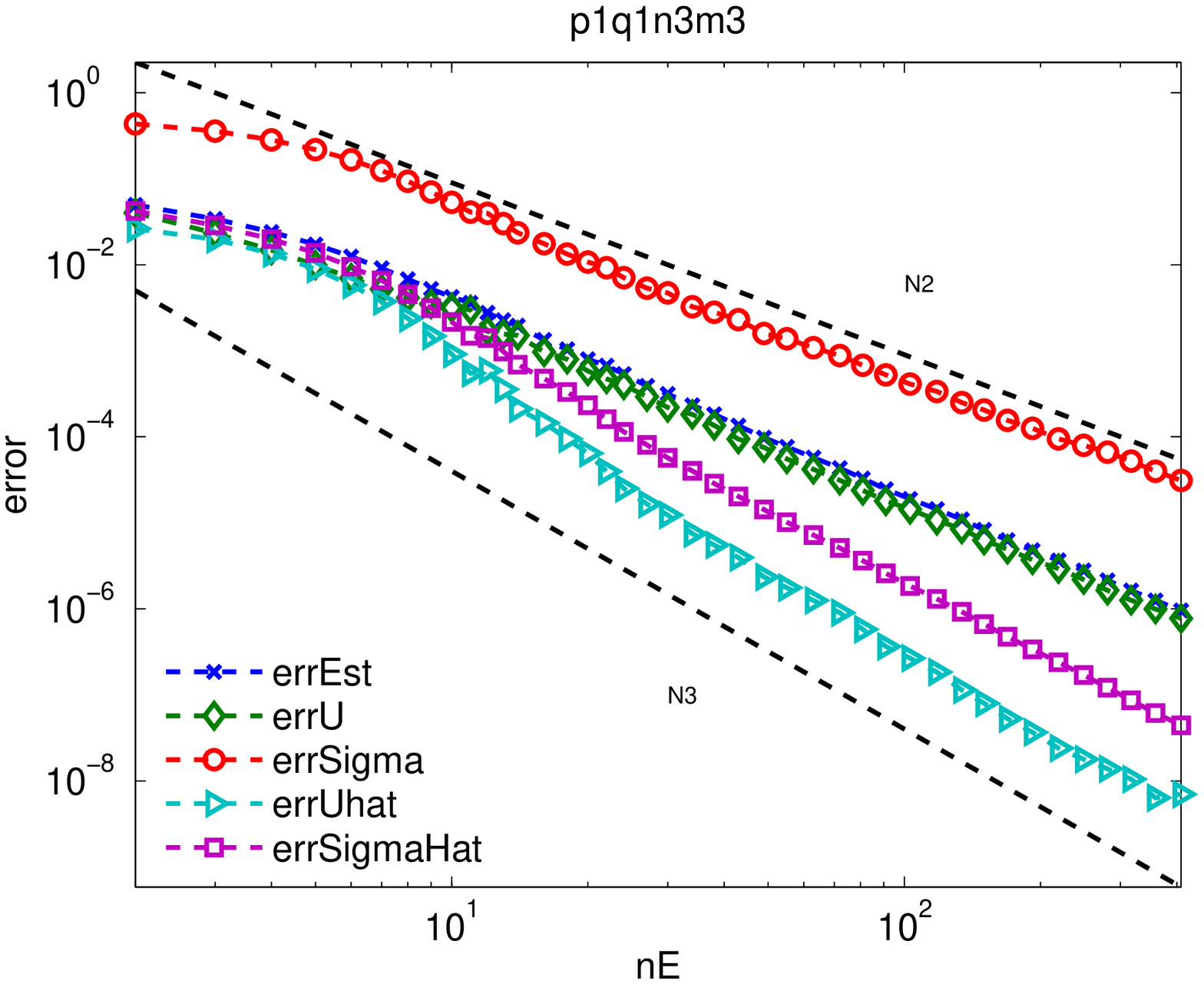}
  \includegraphics[width=0.49\textwidth]{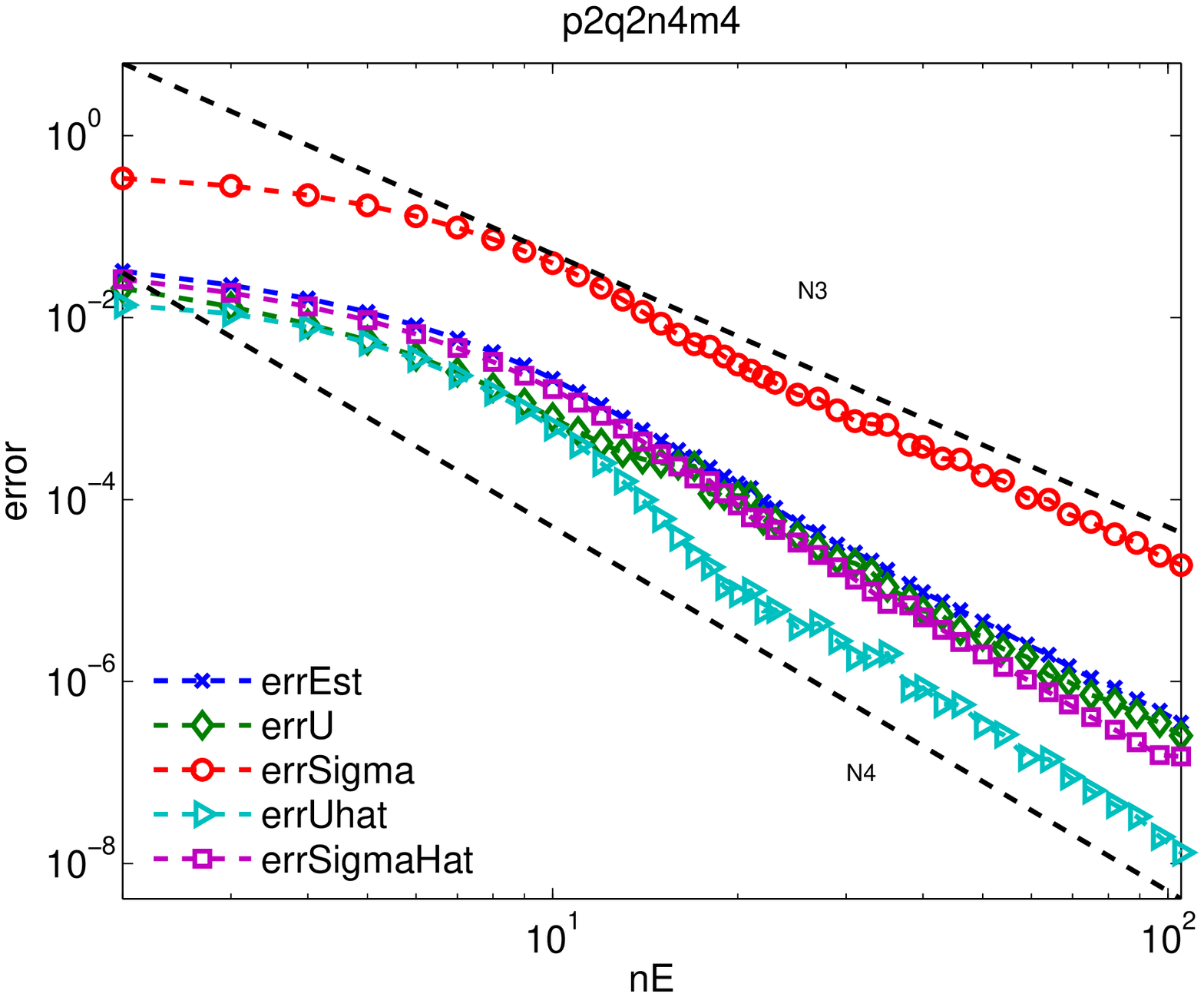}

  \caption{Experimental convergence rates for Example 2 from Section~\ref{sec:example2}.
  Uniform mesh refinement (upper left) and adaptive mesh refinement (upper right and below).}
  \label{fig:example2}
\end{figure}

For the next example we prescribe the exact solution $u(x) = x^\lambda-x$ with $1/2<\lambda<3/2$ on $I = (0,1)$, see
also~\cite[Section~5, Example~3]{ErvinR_NumPDE_06}. 
The right-hand side as well as $\sigma$ are given by
\begin{align*}
  f(x) &= -\frac{\Gamma(\lambda+1)}{\Gamma(\lambda+1-\alpha)} x^{\lambda-\alpha} + \frac{1}{\Gamma(2-\alpha)} x^{1-\alpha}, \\
  \sigma(x) &= Du = \lambda x^{\lambda-1} - 1, \\
  D^{\alpha-2} Du(x) &= \frac{\Gamma(\lambda+1)}{\Gamma(\lambda+2-\alpha)} x^{\lambda+1-\alpha} -
  \frac{1}{\Gamma(3-\alpha)} x^{2-\alpha}.
\end{align*}

We have $u\in H^{\lambda+1/2-\eps}(I)$ and $\sigma\in H^{\lambda-1/2-\eps}(I)$ for all $\eps>0$,
and hence, due to $1/2<\lambda<3/2$, with a view to~\eqref{eq:rates}, we expect
a convergence rate of $\est = \OO(N^{1/2-\lambda})$.
However, with a view to the norm $\norm{\cdot}{U_\alpha}$, the expected rate, dictated by $\sigma$ in this case,
would be $\OO(N^{1/2-\lambda + \alpha/2 - 1})$. This is what we will see for uniform refinement.
In order to regain the optimal convergence orders $\OO(N^{-\min\left( p+1,q+1 \right)})$,
we utilize an adaptive strategy where we use $\est(T)$ as local
refinement indicators and mark elements
$\mathcal{M}\subseteq\TT$ according to D\"orfler's marking criterion
\begin{align}
  \theta \est^2 \leq \sum_{T\in\mathcal{M}} \est(T)^2,
\end{align}
where we use $\theta = 0.4$ and $\mathcal{M}$ is a set of minimal cardinality.
Note that $\theta = 1$ means uniform refinement, i.e., $\mathcal{M} = \TT$.
Each marked element $T\in\mathcal{M}$ is bisected such that local quasi-uniformity
\begin{align*}
  \max_{\substack{T,T'\in\TT \\ \overline{T}\cap\overline{T'}\neq \emptyset} } \frac{\diam(T)}{\diam(T')} \leq 2
\end{align*}
is preserved.
%This can be ensured with, e.g.~\cite{AuradaFFKP_Efficiency_CMAM13}.
Figure~\ref{fig:example2} shows $\est$ and the error quantities for the parameters
\begin{align*}
  \lambda = 0.6, \quad \alpha = 1.2.
\end{align*}
In the upper left plot the results for uniform refinement and $p = q =0$, $n=m=2$ are given.
We observe the convergence rate $\est = \OO(N^{1/2-\lambda+\alpha/2-1}) = \OO(N^{-\lambda + 1/10})$.
As expected, also for the separated error contributions, we observe reduced convergence rates.
Adaptive refinement recovers the optimal rate $\OO(N^{-\min\left( p+1,q+1 \right)})$,
as is seen in the three remaining plots.
As in Example~\ref{sec:example1}, we see that the traces even have better convergence rates.

\subsection{Example~3}\label{sec:example3}

\begin{figure}[htb]
  \centering
  \psfrag{nE}[c][c]{\tiny number of elements}
  \psfrag{error}[c][c]{\tiny error in energy norm}
  \psfrag{N1}[c][c]{\tiny $\OO(N^{-1})$}
  \psfrag{N2}[c][c]{\tiny $\OO(N^{-2})$}
  \psfrag{N3}[c][c]{\tiny $\OO(N^{-3})$}
  \psfrag{N12malpha2}[c][c]{\tiny $\OO(N^{1/2-\alpha/2})$}

  \psfrag{alpha16}[c][c]{\tiny $\alpha = 1.6$}
  \psfrag{alpha18}[c][c]{\tiny $\alpha = 1.8$}

  \psfrag{pq0Unif}{\tiny $p=q=0, \theta = 1$}
  \psfrag{pq0adap}{\tiny $p=q=0,\theta = 0.4$}
  \psfrag{pq1}{\tiny $p=q=1, \theta = 0.4$}
  \psfrag{pq2}{\tiny $p=q=2, \theta = 0.4$}

  \includegraphics[width=0.49\textwidth]{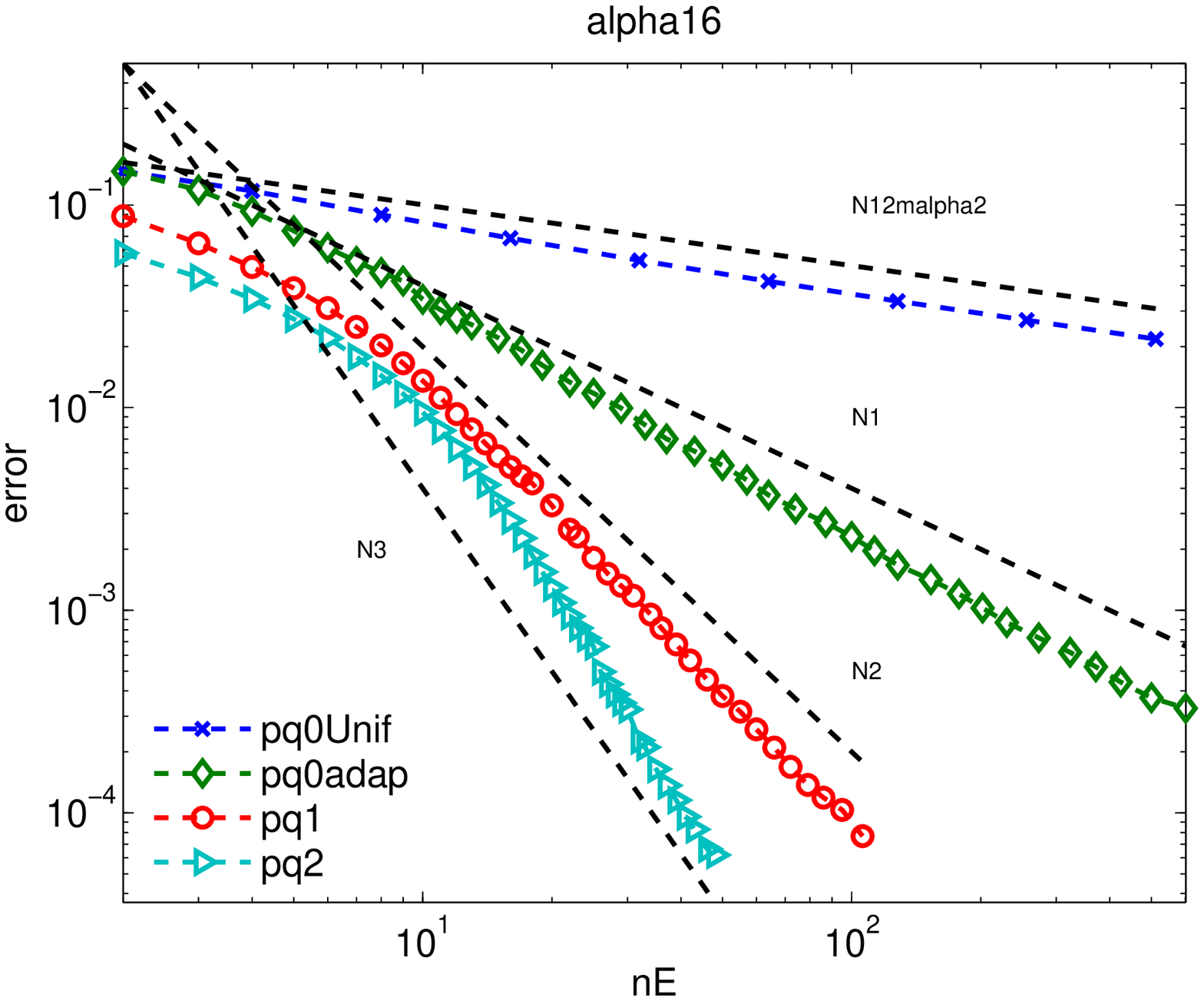}
  \includegraphics[width=0.49\textwidth]{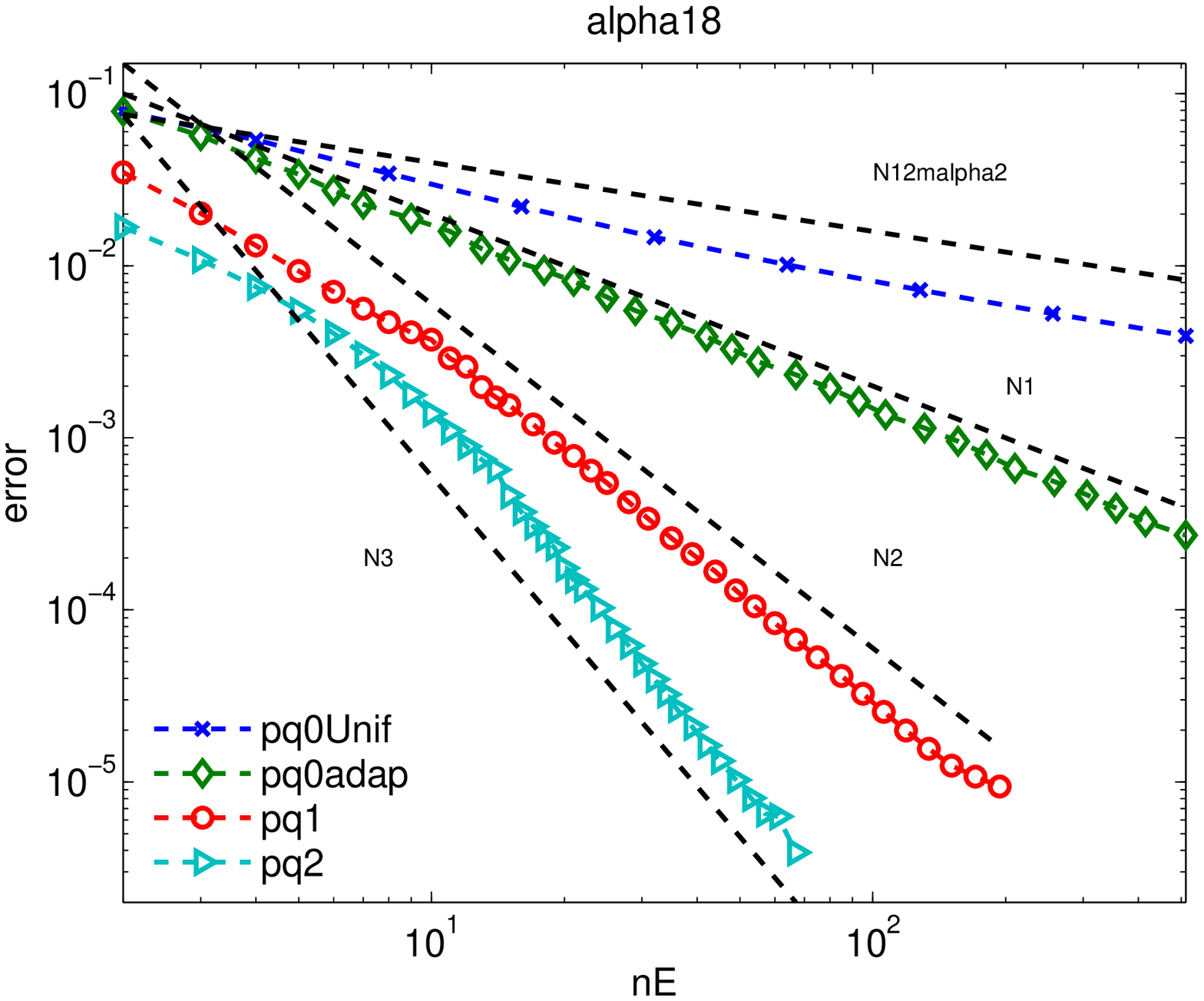}

  \caption{Experimental convergence rates for Example 3 from Section~\ref{sec:example3}.
  The choice $\theta=1$ refers to uniform mesh refinement, while $\theta=0.4$ refers
to adaptive mesh refinement.}
  \label{fig:example3}
\end{figure}

In the last experiment we set $f(x) := \log(x)$ for $x\in I =(0,1)$ and note that
$f\in L_2(I)$.
For this right-hand side we do not know the explicit form of the solution $u$.
Therefore, we only plot the error in the energy norm $\est$ for different values of $p,q$, $m,n$ and $\alpha$, respectively.
Throughout, we set $p=q$ as well as $m = n = p+2$.
Figure~\ref{fig:example3} shows the error in the energy norm $\est$ for $\alpha = 1.6$ (left) and $\alpha =1.8$ (right).
We compare uniform refinement ($\theta=1$) and adaptive refinement with $\theta=0.4$ for $p=q=0$.
Moreover, we plot the results in the adaptive case with $p=q=1$ resp. $p=q=2$.
We observe that for adaptive refinement we obtain convergence rates $p+1$, i.e., $\est = \OO(N^{-(p+1)})$, whereas for
uniform refinement we get only the suboptimal rate $\alpha/2-1/2$.

\bibliographystyle{abbrv}
\bibliography{literature}
\end{document}